\documentclass{wart}

\suppldatatrue

\definecolor{green}{RGB}{60,100,20}

\newtheorem{theorem}{Theorem}
\newcounter{other}
\numberwithin{other}{section}
\newtheorem{cor}[other]{Corollary}
\newtheorem{proposition}[other]{Proposition}
\newtheorem{lemma}[other]{Lemma}
\theoremstyle{definition}
\newtheorem{remark}[other]{Remark}
\newtheorem{definition}[other]{Definition}

\numberwithin{equation}{section}

\title{Dominating Sets in Bergman Spaces on Strongly Pseudoconvex Domains}
\author{A. Walton Green and Nathan A. Wagner\footnote{Supported by NSF GRF, grant number DGE-1745038}}

\begin{document}

\maketitle

\begin{abstract}
We obtain local estimates, also called propagation of smallness or Remez-type inequalities, for analytic functions in several variables. Using Carleman estimates, we obtain a three sphere-type inequality, where the outer two spheres can be any sets satisfying a boundary separation property, and the inner sphere can be any set of positive Lebesgue measure. We apply this local result to characterize the dominating sets for Bergman spaces on strongly pseudoconvex domains in terms of a density condition or a testing condition on the reproducing kernels. Our methods also yield a sufficient condition for arbitrary domains and lower-dimensional sets.
\end{abstract}


\section{Introduction} 
Let $\mathcal F$ be a function space defined on a metric-measure space $(\Omega,d,\mu)$. Finding the so-called dominating sets for $\mathcal F$ is to find $E \subset \Omega$ such that functions in $\mathcal F$ can be continuously reconstructed from their values on $E$. In other words, the restriction map $f \mapsto f|_E$ is invertible.

Such questions have been considered by many authors, having complete solutions in many classical function spaces---see the survey \cite{fricain2015survey}. Recently, there has been renewed interest in uncertainty principle versions of this problem due to the applications in control theory. In such a case, $\mathcal F$ is defined by some sort of Fourier decay or support condition. We take a different perspective here and consider the Bergman spaces, defined below.

When the underlying metric measure space $(\Omega,d,\mu)$ of a function space $\F \subset L^p(\Omega, d\mu)$ is acted on transitively and invariantly by a group, many simplifications can be made. It is not too hard to check that the following \textit{relative density} conditions are equivalent and necessary conditions for $E$ to be a dominating set:
	\begin{equation}\label{eq:intro-dense} \inf_{z \in \Omega} \mu(E \cap B_z) > 0, \quad \mbox{and} \quad \inf_{z \in \Omega} \|\phi_z \bigr|_E\| > 0 \end{equation}
for some ball $B$ or some non-zero function $\phi$. $B_z$ or $\phi_z$ mean $B$ or $\phi$ translated, by the group action, to the point $z \in \Omega$.

This is the first novel feature of our paper. The problem of dominating sets in the homogeneous Bergman spaces (those whose domains have a transitive automorphism group) was completely solved by Luecking almost 40 years ago \cite{luecking1981inequalities,luecking1984closed}, showing that relative density (\ref{eq:intro-dense}) is also sufficient. In such a setting, one can use the Euclidean geometry (say far from the boundary) and then using the automorphisms of the domain, connect this to the invariant complex geometry. Such features also enter into the recent work of Hartmann et. al. in \cite{hartmann2021dominating}, where they refine Luecking's results on the disc and obtain the sharp form of the \textit{sampling constant}, polynomial in terms of the lower bound (\ref{eq:intro-dense}), which is related to the norm of the inverse operator $f\bigr|_E \mapsto f$.

Our main goal here is to extend both of these results to more general domains, with little to no automorphic structure. In general domains, there are many so-called invariant metrics, so it is not immediate what sort of measure-theoretic density condition is necessary. However, one which is both necessary and which acknowledges the invariant complex geometry can be given by the Berezin transform, which is defined in (\ref{eq:intro-berezin}) below.

 To streamline notation, unless otherwise specified, the letter $C$ will denote a constant that only depends on the domain $\Omega$ and can possibly change from line to line.

\subsection{Main Result}
To state our main result, let us introduce some definitions. We will use the notation $\Hol(X;Y)$ to denote the space of functions which are holomorphic on $X$, taking values in $Y$. When $Y = \C$, we use $\Hol(X)$.
For a domain $\Omega \subset \C^n$, $1 \le p \le \infty$, and $\alpha > -1$, define the power-weighted Bergman spaces 
	\[ A^p_\alpha(\Omega) = \{ f \in \Hol(\Omega) : \int_\Omega |f(z)|^p |\rho(z)|^\alpha dA(z)<\infty \}. \]
$\rho(z)$ is a defining function of the domain $\Omega$ which means $\Omega = \{ \rho < 0\}$. $dA$ denotes the volume element on $\C^n \equiv \R^{2n}$. It is straightforward to show that $A^p_{\alpha}(\Omega)$ is a Banach space and is a closed subspace of $L^p_{\alpha}(\Omega),$ which is the weighted $L^p$ space on $\Omega$ with weight $|\rho|^\alpha.$ Bergman spaces with radial weights have been extensively studied (see for example \cite{peloso1994}, \cite{cuc2018}, \cite{zhu2016}). 

$A^2_\alpha(\Omega)$ is a reproducing kernel Hilbert space. Let $K_\alpha(z,w)$ denote its reproducing kernel and $k_z^{p,\alpha}$ the $L^p_\alpha$ normalization at $z \in \Omega$:
$$k^{p,\alpha}_z(w)=\frac{K_\alpha(z,w)}{\|K_\alpha(z,\cdot)\|_{L^p_{\alpha}(\Omega)}}.$$
The following quantity will be of crucial importance in testing if a set $E$ is a dominating set,
	\begin{equation}\label{eq:intro-berezin} \tilde T_E^{p,\alpha}(z) = \|k_z^{p,\alpha}\|_{L^p_\alpha(E)}. \end{equation}
When $p=2$, this is the well-known \textit{Berezin transform} of the Toeplitz operator with symbol $1_E$. Broadly, our main result states that for $E$ to be a dominating set for $A^p_\alpha(\Omega)$, it is enough for $\tilde T_E^{p,\alpha}$ to not vanish on $\partial \Omega$.

\begin{theorem}\label{thm:main}
Let $\Omega$ be a smoothly bounded strongly pseudoconvex domain and $Y(w,r)$ be a ball in the Kobayashi metric on $\Omega$ of radius $\tanh^{-1}r$ centered at $w$ (see Definition \ref{def:pseudo}). Then, for any $E \subset \Omega$, the following are equivalent.
\begin{itemize}
	\item[(i)] For any  $1 \leq  p \leq \infty$ , $\alpha>-1$, there exists $C>0$ such that
	\[ \|f\|_{A^p_\alpha(\Omega)} \le C \|f\|_{L^p_\alpha(E)} \]
	\item[(ii)] $E$ is relatively dense, which means there exists $r>0$ such that
	\[ \inf_{w \in \Omega} \frac{|E \cap Y(w,r)|}{|Y(w,r)|} > 0. \]
	\item[(iii)] For some $1 < p < \infty$, and  non-negative integer  $\alpha$, 
	\[ \inf_{z \in \Omega} \tilde T_E^{p,\alpha}(z) > 0. \]
\end{itemize}
\end{theorem}
Furthermore, if the infimum in (ii) is at least $\gamma$, then the form of the constant in ``(ii) implies (i)'' is $C \gamma^{-q}$ for $C,q>0$ depending on $r,p,\alpha,\Omega$, but not on $E$. The same holds in ``(iii) implies (i)'' if we have a lower bound $\tilde T_E^{p,\alpha}(z) \ge \gamma$. We also mention that this sharp dependence (polynomial in $\gamma$) was recently obtained for $\Omega=\mathbb D$ in \cite{hartmann2021dominating}, with more information concerning the other parameters as well.

One application of this theorem is to the reproducing kernel hypothesis (RKH), which concerns the connection between an operator $T$ and its Berezin transform. For nonnegative function $\sigma$, let $T_\sigma: A^2_\alpha(\Omega) \to A^2_\alpha(\Omega)$ be the Toeplitz operator defined by
	\[ T_\sigma^\alpha(f) = P_\alpha(\sigma f) \]
where $P_\alpha$ is the orthogonal projection from $L^2_\alpha(\Omega)$ to $A^2_\alpha(\Omega)$. The Berezin transform of such an operator is
	\[ \tilde T_\sigma^\alpha(z) = \ip{T_{\sigma}^{\alpha} k_z^{2,\alpha}}{k_z^{2,\alpha}}. \]
This coincides with the definition (\ref{eq:intro-berezin}) when $p=2$ and $\sigma = 1_E$.
\begin{cor}\label{cor:toeplitz}
Let $\alpha$ be a non-negative integer , $\sigma \in L^\infty(\Omega)$, $\gamma>0$. There exists $c>0$ such that
	\[ \ip{T_\sigma^\alpha f}{f} \ge c\|f\|^2 \]
if and only if 
	\[ \liminf_{z \to \bomega} \tilde T_\sigma^\alpha (z) >0.\]
Moreover, there exists $q>0$ such that if $\tilde T_\sigma^\alpha(z) \ge \gamma$, then
	\[ \ip{T_\sigma^\alpha f}{f} \ge c \gamma^{q}\|f\|^2. \]
\end{cor}

This can be viewed as a version of the RKH for invertibility of positive Toeplitz operators. This also implies a version of the RKH for boundedness. Initially one may hope to obtain $\|T_\sigma\| \le C \sup_{z \in \Omega} |\tilde T_\sigma^{\alpha}(z)|$ where $C>1$. However, if $\|\sigma\|_{L^\infty} \le 1$, then this gives no improvement on $\|T_\sigma\| \le \|\sigma\|_\infty$ unless $\tilde T_\sigma^\alpha(z)$ is \textit{very} small. Applying some elementary functional analysis to Corollary \ref{cor:toeplitz}, one obtains the following.

\begin{cor}
 Let $\alpha$ and $\sigma$ be as above . There exists $0<c<1$ and $q>0$ such that if $\tilde T_\sigma^\alpha(z) \le \|\sigma\|_\infty(1-\gamma)$, then
	\[ \|T_{\sigma}^\alpha f\| \le \|\sigma\|_\infty (1-(c\gamma)^q)\|f\|. \]
\end{cor}

\subsection{Reverse Carleson Measures} \label{subsec:reverse}
Dominating sets are a special class of \textit{reverse Carleson measures}, which are measures $\mu$ such that for all $f \in A^p_\alpha(\Omega)$,
	\[ \|f\|_{L^{p}_\alpha(\Omega)} \le C_\mu \|f\|_{L^{p}_\alpha(\Omega,\mu)}, \quad \|f\|_{L^p_\alpha(\Omega,\mu)} = \left( \int_\Omega |f(z)|^p |\rho(z)|^\alpha \, d\mu(z) \right)^{1/p}. \]
Even when $\Omega=\mathbb D$, there is not a complete characterization of such measures \cite{fricain2015survey}. Sufficient density conditions have been given by Luecking \cite{luecking1985forward} and were recently extended by Calzi and Peloso \cite{calzi2021carleson} to the quite general case when $\Omega$ is a homogeneous type II Siegel domain.  The measures $\mu$ for which the $L^p_\alpha(\Omega)$ and $L^p_\alpha(\Omega,\mu)$ norms are equivalent on $A^p_\alpha(\Omega)$ can be characterized in some terms of the automorphisms of the domain $\Omega$ and the zero sets of functions $f \in A^p_\alpha(\Omega)$. However, we have already mentioned the lack of automorphic structure in our setting, and furthermore, the zero sets of functions in $A^p_\alpha(\Omega)$ do not have a measure-theoretic characterization---see the discussion below.

Our main result above characterizes the reverse Carleson measures of the form $d\mu = 1_E \, dA$. The methods also apply to
	\[ d \mu = 1_E \, d\H^{2n-2+\nu}, \quad \nu>0 \]
where $\H^s$ is the $s$-dimensional Hausdorff measure.

\begin{theorem}\label{thm:lower-dim-kob}
Let $\Omega \subset \mathbb{C}^n$ be a smoothly bounded strongly pseudoconvex domain and $1 \leq p \leq \infty.$  Suppose $E \subset \Omega$ satsifies, for some $r,\gamma >0$, $\nu \in (0,2)$
	\begin{equation}\label{eq:density-kob} \frac{|Y(w,r)|^{\frac{n-1}{n(n+1)}}\H^{2n-2+\nu}(E \cap Y(w,r))}{|Y(w,r)|^{(2n-2+\nu)/2n}} \ge \gamma \end{equation}
for all $w \in \Omega$. Then, 
	\[ \|f\|_{L^p_\alpha(\Omega)} \le C \max\{C,\gamma^{-q}\} \|f\|_{L^p_\alpha(E,\H^{2n-2+\nu})} \]
for all $f \in A^p_\alpha(\Omega)$.
\end{theorem}

It is important to point out that this result \textit{does not} hold when $\nu=0$. This is because the zero sets of holomorphic functions of $n$ variables are $(2n-2)$-dimensional. So in this case, the only hope of a measure-theoretic condition is that the density (\ref{eq:density-kob}) be large enough, which can be derived from our methods, see Corollary \ref{cor:2n-2}. One can go below the threshold of $2n-2$, see Theorem \ref{reversecarleson} for the complete statement. Our main application though, will be the following sampling theorem.

\begin{theorem}\label{thm:sampling}
Let $\Omega \subset \mathbb C^n$ be a smoothly bounded strongly pseudoconvex domain. For each $r>0$, $1 \le p < \infty$, and $\alpha > -1$, there exists $C,q>0$ such that if $s < C^{-1}\gamma^q$ and $\{a_j\}$ is any sequence in $\Omega$ satisfying
	\[ \inf_{z \in \Omega} \frac{\left|\cup_{j =1}^\infty Y(a_j, s) \cap Y(z,r)\right|}{|Y(z,r)|} \ge \gamma, \]
then
	\[ \|f\|_{L^p_\alpha(\Omega)}^p \le C\sum_{j=1}^\infty |f(a_j)|^p\operatorname{dist}(a_j,\partial \Omega)^{n+1+\alpha} \]
for all $f \in A^p_\alpha(\Omega)$.
\end{theorem}

Since the conditions in Theorems \ref{thm:lower-dim-kob} and \ref{thm:sampling} may not be necessary, one may prefer to test over geometrically simpler sets than the Kobayashi balls, for example Euclidean balls or cubes. This is indeed possible.

\begin{theorem}\label{thm:lower-dim-euclid}
Let $\Omega \subset \C^n$ be open. Suppose $E \subset \Omega$ satsifies, for some $r,\gamma >0$, $\nu \in (0,2]$,
	\begin{equation}\label{eq:density-euclid} \frac{\H^{2n-2+\nu}(E \cap Q)}{\ell(Q)^{2n-2+\nu}} \ge \gamma \end{equation}
for all cubes $Q \subset \Omega$ of side length $r \dist(Q, \Omega^c)$. Then, for any $1 \le p \le \infty$, $\alpha > -1$, there exist $C,q>0$ such that
	\[ \|f\|_{L^p_\alpha(\Omega)} \le C \max\{C,\gamma^{-q}\} \|f\|_{L^p_\alpha(E,\H^{2n-2+\nu})} \]
for all $f \in A^p_\alpha(\Omega)$.
\end{theorem}

The analogue holds for Theorem \ref{reversecarleson} also.

\subsection{Orientation}
The strategy we employ first shows that one can get a sufficient condition by completely ignoring the complex geometry of the domain, treating the analytic function as a solution to the $\dbar$ equation in $\Omega \subset \R^{2n}$. First, we use a Carleman estimate to prove a local Remez-type inequality uniform over many sets with a certain boundary separation. This is Section \ref{sec:local}. In Section \ref{sec:pseudo}, this is applied by decomposing the domain into Kobayashi balls. Section \ref{subsec:suff} proves (ii) implies (i) in Theorem \ref{thm:main}. The other piece of Theorem \ref{thm:main}, the connection between relative density and the Berezin transform, is established in Section \ref{subsec:nec}. It relies on further understanding the geometry of Kobayashi balls in order to obtain a vanishing Rudin-Forelli estimate, which is the content of Section \ref{subsec:kernel}. Lastly, we extend to lower-dimensional sets in Section \ref{sec:lower-dim}, modeled after the work of Logunov and Malinnikova in harmonic functions \cite{logunov2018quantitative}.

The authors would like to thank Marco M. Peloso for some helpful comments.  They would also like to acknowledge the anonymous referees for their extremely thorough and careful review of the manuscript and helpful suggestions, which led to many improvements.

\section{Local Estimate for Holomorphic functions of several variables}\label{sec:local}
In this section, we will establish a three sphere-type inequality of the form
	\begin{equation}\label{eq:three} \sup_{Y} |f| \le (\sup_E |f|)^\theta (\sup_X |f|)^{1-\theta}, \quad f \mbox{ holomorphic on }X \end{equation}
where $X$ is a suitable ``double'' of $Y$ and $E \subset Y$ of positive measure. Past results of this type rely on any number of features; see \cite{brudnyi1999local,nazarov2004lower,logunov2018quantitative,kovrijkine01,lebeau19} for five different perspectives. Despite the different approaches, each result is proved when $X$ and $Y$ have special geometry---usually both are Euclidean balls. The result which appears to be most readily extended to other geometries is the result of Lebeau and Moyano \cite{lebeau19} which uses Carleman estimates.

This technique, which in recent history has dominated the field of unique continuation for PDEs, has strong historical connections to complex analysis and specifically pseudoconvex domains. In this regard, we mention the pioneering work of T. Carleman \cite{carleman1939probleme} on elliptic PDE and the remarkable development of his ideas by H\"ormander \cite{hormander1973introduction,hormander-books} in both real and complex analysis.

Our method for establishing (\ref{eq:three}) is a refinement of Section 3 in \cite{lebeau19}. The first main difference is we would like to obtain the result in arbitrary dimensions. This requires us to consider each component of the $\dbar$ operator. Second, if one only needed the result for two sets $X$ and $Y$ with regular boundary, the result from \cite{lebeau19} applies. However, we will need the result to hold uniformly over many sets with varying geometry. 

Examining their proof, it turns out that the constants one obtains depend only on the estimates for the Green functions associated to the larger set $X$, specifically, that 
	\[ \sup_{y \in Y} G(x,y) \to 0\mbox{ as } x \to \bomega,\mbox{ and } \inf_{x,y \in Y} G(x,y) > 0. \]
To those well-versed in elliptic PDEs, it may be obvious that the convergence rate and lower bounds depend only on the regularity of the boundary $X$ and the separation between $\partial X$ and $Y$. This principle is well-known, but in some sense has remained unexamined until very recently \cite{david2021carleson}.

The modification we propose dispenses with the explicit regularity of $\partial X$ and instead relies on the regularity of an intermediate domain $Z$ which we will construct. Some definitions are now in order. For any $s>0$ and $\X \subset \C^n$, let us introduce the notation
	\begin{equation}\label{eq:nbhd} \X_s = \{ x \in \X : \dist(x,\partial \X) \le s \}, \quad \X_s' = \X \backslash \X_s. \end{equation}

\begin{definition}
Let $\mathcal \X$ be a collection of bounded, open sets in $\R^n$. We say $\mathcal \X$ is a \textit{regular boundary family} if for each $\X \in \mathcal \X$ and $p \in \partial \X$, there exists a barrier function $\omega_p^\X$ satisfying the following:
\begin{itemize}
	\item[(i)] For each $\eta>0$,
	\begin{equation}\label{eq:barrier-pos} \inf_{\X \in \mathcal \X} \inf_{p \in \partial \X} \inf_{ y \in \bar Z \setminus B(p, \eta) } \omega_p^\X(y) > 0. \end{equation}
	\item[(ii)] For each $\ep>0$, there exists $\delta>0$ such that for any $\X \in \mathcal \X$,
	\begin{equation}\label{eq:barrier-zero} \omega_p^\X(y) \le \ep \quad \mbox{whenever } |y-p| < \delta.\end{equation}
\end{itemize}
A \textit{barrier function} at $p \in \partial \X$ is a nonegative function on $\partial \X$ which vanishes only at $p$ and is superharmonic in $\X$. 
\end{definition}

It will be important for us that the following class of domains forms a regular boundary family.
\begin{definition}
For $d>0$, a domain $\X$ satisfies the \textit{$d$-uniform exterior sphere condition} if at each point $p \in \partial \X$, there exists an open ball $B$ of radius $d$ such that
	\[ \bar B \cap \bar Z = \{p\}. \]
\end{definition}

For each $d>0$, the collection of all domains satisfying the $d$-uniform exterior sphere condition forms a regular boundary family. In fact, one can explicitly construct the barrier functions in this case, see \cite[p. 27]{gilbarg2001elliptic}.

\begin{definition}
Let $\Gamma$ be the fundamental solution to the Poisson equation in $\R^{2n}$ and $\X \subset \R^{2n}$. We say $G_\X(x,y)$ is the \textit{Green function} for $\X$ if
	\[ G_\X(x,y) = \Gamma(x-y) - H_\X(x,y) \]
where $H_\X(x,y)$ is the harmonic corrector satisfying, for each $x \in \X$,
	\[ \left\{ \begin{array}{rcll} \Delta_y H_\X(x,y) &=&0&y \in \X, \\ H_\X(x,y) &=& \Gamma(x-y) & y \in \partial \X \end{array} \right. \]
\end{definition}

\begin{proposition}\label{prop:green}
Let $\mathcal \X$ be a regular boundary family such that $\diam \X \le \frac 12$ for each $\X \in \mathcal \X$. Let $G_\X$ be the Green function for each $\X \in \mathcal \X$. Then, for each $\eta>0$,
	\[ \inf_{\X \in \mathcal \X} \inf_{\substack{x,y \in \X_\eta', \\ x \ne y }} G_\X(x,y) > 0,\]
and for each $\ep>0$, there exists $\delta>0$ such that for any $\X \in \mathcal \X$,
	\[ \sup_{y \in \X_\eta'} G_\X(x,y) \le \ep \quad \mbox{whenever } x \in \X_\delta. \]
\end{proposition}
We will give an elementary proof of Proposition \ref{prop:green} at the end of this section. However, if one is content to accept this, we can obtain the following proposition extending the result of Lebeau and Moyano in \cite{lebeau19}.

\begin{proposition}\label{prop:local}
Let $d,\ell>0.$ There exists $C>0$ such that for any measurable subsets $Y \subset X \subset \mathbb{C}^n$, and $E \subset Y$ of positive measure, if the sets $X$, $Y$ satisfy, for some affine map $D$,
	\begin{itemize}
	\item[(i)] $\diam( D(Y)) = 1/4$,
	\item[(ii)] $|D(Y)| \ge \ell$,
	\item[(iii)] $\dist(\partial D(X),\partial D(Y)) \ge d$,
	\end{itemize}
then
\begin{equation}\label{eq:prop-local} \avg{f}_{2,Y} \le Ce^{ C(N+1)S^*} \avg{f}_{2,E} \end{equation}
for all $f \in A^2(X)$, where 
	\[ \avg{E}_Y= \frac{|E|}{|Y|}, \quad \avg{f}_{p,F} = \frac{\|f\|_{L^p(F)}}{|F|^{1/p}}, \quad N = \log \frac{ \|f\|_{L^2(X)}}{\|f\|_{L^2(Y)}}, \quad S^* = \left\{ \begin{array}{cl} 1+\log \frac{1}{\avg{E}_Y} & n=1; \\ 1+\avg{E}_Y^{\frac 1n -1} & n \ge 2. \end{array} \right. \]
\end{proposition}

Let us also clarify that by affine, we mean some combination of translation, dilation, and rotation of the $n$ complex variables $(z_1,z_2,\ldots,z_n)$. By the invariance of the conclusion under such affine changes of variables, we can assume $X$ and $Y$ themselves satisfy (i)--(iii). As long as $d < 1/4$, we can also assume that $\diam X \le \tfrac 12$ by replacing $X$ with $X$ intersected with a ball of radius $\tfrac 12$. If $d > \frac 14$ then a much simpler proof can be given which uses explicit estimates on the fundamental solution $\Gamma$ instead of the Green functions (we will not consider this case in the proof). Let us give the proof of Proposition \ref{prop:local}, assuming Proposition \ref{prop:green} for the time being. 

\subsection{Carleman Estimate}\label{subsec:carleman}
Recall that a holomorphic function $f:\C^n \to \C$ is holomorphic in each complex variable and therefore it satisfies the $\dbar$ equation in each variable. Denote elements in $\C^n$ by $z=(z_1,\ldots,z_n)=(x_1,x_2,\ldots,x_{2n}) \in \R^{2n}$, where $z_k=x_{2k-1}+ix_{2k}$ for $1 \leq k \leq n$, and by $\dbar_k$ the $k$-th component of the $\dbar$ operator,
	\[ \dbar_k = \frac 12 \left(\frac{\partial}{\partial x_{2k-1}} + i\frac{\partial}{\partial x_{2k}}\right). \]
We have that $\dbar_k f =0$ for each $k$. We will first prove a Carleman estimate for each $\dbar_k$.

Fix $Z \subset \C^n$ and $\phi:Z \to \mathbb R$ with $\phi, \Delta \phi \in L^\infty$. For a parameter $h>0$, define the differential operator on $u \in C^\infty_0(Z)$, $P_h u = e^{\phi/h} \frac{2h}{i} \dbar_k (e^{-\phi/h} u)$. Then, we can compute
	\begin{align*} P_h &= (-\phi_{x_{2k}} + h\partial_{x_{2k}})-i(-\phi_{x_{2k-1}} + h\partial_{x_{2k-1}}) \\
			&= (\frac{h}{i}\partial_{x_{2k-1}}-\phi_{x_{2k}}) +i(\frac{h}{i}\partial_{x_{2k}}+\phi_{x_{2k-1}})\\
			&=:A+iB.
	\end{align*}
 Integration by parts yields
	\[ \|P_h u \|^2 = \|A u \|^2 + \|B u \|^2 + i\ip{[A,B] u}{u}, \]
 where the norms and inner product in the above display are those of $L^2(Z).$ 

It is simple to compute $[A,B] = -ih(\partial^2_{x_{2k-1}}\phi + \partial^2_{x_{2k}}\phi)$. Taking {$u=e^{\phi/h}g$} for some $g \in C^\infty_0$, we obtain
	\[ 4h^2\int_Z e^{2\phi/h} |\dbar_k g|^2 \ge h \int_Z e^{2 \phi/h}|g|^2 (\partial^2_{x_{2k-1}}\phi + \partial^2_{x_{2k}}\phi). \]
Summing over all $k$ and picking $g=\psi f$ where $f$ is holomorphic in $Z$ and $\psi \in C_0^\infty(Z)$ with $\psi = 1$ on $\bar Y \subset Z$, we obtain, with $M=\sum_{k=1}^n \|\dbar_k \psi\|_\infty^2$,
	\[ 4Mh^2 \int_{Z \backslash \{\psi=1\}} e^{2\phi/h} |f|^2 \ge h \int_Y e^{2 \phi/h}|f|^2 \Delta \phi. \]
We used the fact  that $\dbar_k (\psi f) = (\dbar_k \psi) f$ since $f$ is holomorphic.

\subsection{Choice of $Z$ and $\phi$}
Cover $\partial X$ with a collection $\mathcal B$ of open balls centered on $\partial X$ with radius $d/2$. Set
	\[ Z = X \, \backslash \bigcup_{B \in \mathcal B} B. \]
$Z$ belongs to the regular boundary family
	\[ \mathcal \X_{d/4} := \{ \X : \X \mbox{ satisfies the $d/4$-uniform exterior sphere condition}\}. \]
Indeed, if $q \in \partial Z$, then $q \in \partial B_q$ for some $B_q$. However, we can fit a ball $B^*_q$ of radius $d/4$ inside $B_q$ such that $\bar B^*_q \cap B^c_q = \{ q\}$. Therefore,
	\[ \{q\} \subset \bar B^*_q \cap \bar Z = \bar B^*_q \cap X \cap_{B \in \mathcal B} B^c \subset \bar B^*_q \cap B_q^c = \{q\}. \]

We are ready to construct $\phi$. Set $\phi = \phi_1-\rho \phi_2$ where for $j=1,2$,
	\[ \left\{\begin{array}{rcll}  -\Delta \phi_j &=& \tilde \chi_j &\mbox{in } Z; \\
							\phi_j &=& 0 &\mbox{on } \partial Z; \end{array} \right. \tilde\chi_1=\frac{\chi_E}{\avg{E}_Y}, \quad \tilde\chi_2=\chi_Y.\]
On the face of it, $\phi$ only satisfies the regularity assumption $\Delta \phi \in L^\infty$. However, we will obtain precise estimates on $\phi$ subsequently which easily imply $\phi \in L^\infty$.
In order to utilize the exponential weights in the Carleman estimate, we want $\phi$ to be smaller on $Z \backslash \{\psi=1\}$ than on $Y$. 

First, since $\phi_j(x) = \int_{Y} G_Z(x,y)\tilde \chi_j(y) \, dy$, by Proposition \ref{prop:green} and assumption (ii), there exists $c_1$ depending only $d$ such that
	\[ \inf_Y \phi_1 = c_1>0.\]
Furthermore, since $Z \subset X \subset B(z,1/2)$ for some $z \in \mathbb{C}^n$ (recall we are assuming that $\diam X \le \frac{1}{2}$), the fundamental solution, $\Gamma$, is positive on the boundary, so by the maximum principle $\phi_j \le \Gamma * \tilde \chi_j$. This upper bound is maximized when $E$ and $Y$ are balls in which case the integral can be estimated above by
	\begin{equation} \sup_Y \phi_1 \le C\left\{ \begin{array}{cl} 1+\log \frac{1}{\avg{E}_Y} & n=1 \\ \avg{E}_Y^{\frac 1n -1} & n \ge 2 \end{array} \right. , \quad \sup_Y \phi_2 \le C\label{eq:s-star}\end{equation}
for some $C$ depending on $n$. Therefore, picking $\rho=c_1/2C$, $\inf_Y \phi \ge c_1/2$.
Using the second property of the Green function in Proposition \ref{prop:green}, there exists $s$ (depending only on $d$ and $c_1$) such that $\phi(x) \le c_1/4$ if $x \in Z_s$.

Pick the cutoff $\psi \in C_0^\infty(Z)$ from above to be equal to 1 on $Z_s'$. Therefore we obtain, for $S = \sup_{Y} \phi$ and $\delta=c_1/2$,
	\begin{equation}\label{eq:star} Ch e^{\delta /h} \int_X |f|^2 + \avg{E}_Y^{-1} e^{2S/h} \int_E |f|^2 \ge \rho e^{2\delta/h} \int_Y |f|^2. \end{equation}
This can be turned into the desired product form using a now standard trick of optimizing in $h$.
For ease of notation, set 
	\[ \nu = 2(S-\delta), \quad \theta = \frac{\delta}{\nu +\delta}, \]
	\[ A=\int_X |f|^2, \quad B= \avg{E}_Y^{-1}\int_E |f|^2.\]
$\nu$ is positive which can be seen from the definitions of $S$ and $\delta$.
(\ref{eq:star}) reads
	\[ C he^{-\delta/h}A + e^{\nu/h} B \ge \rho \int_Y |f|^2.\]
Set $G(h) = (Ch)^{-1}\exp(\frac{\nu+\delta}{h})$. $G$ is decreasing and takes all values in $(0,\infty)$. Assume $B \ne 0$ and pick $h_0$ such that $G(h_0)=\frac AB$. If $h_0 \ge 1$, then
	\[ \int_Y |f|^2 \le A = G(h_0)^\theta B^\theta A^{1-\theta} \le G(1)^\theta B^\theta A^{1-\theta} = C^{-\theta} e^\delta B^\theta A^{1-\theta}. \]
On the other hand, if $h_0 < 1$, then we have
	\[ B\exp(\nu /h_0)= \exp(-\delta/h_0)Ch_0 A \le \exp(-\delta/h_0)CA. \]
Therefore,
	\begin{align*} \rho \int_Y |f|^2 &\le e^{\nu/h_0} B + Ch_0e^{-\delta/h_0}A \\
			&= 2\exp(\nu/h_0) B \\
		&\le 2 \exp(\nu/h_0)^\theta B^\theta (\exp(-\delta/h_0)CA)^{1-\theta} \\
		&= 2C^{1-\theta}B^\theta A^{1-\theta}
	\end{align*}
verifying in the last line that $\nu\theta-\delta+\delta\theta = 0$. Elementary manipulations yield the final form of the inequality (\ref{eq:prop-local}).

\subsection{Green function estimates}
Let us now prove Proposition \ref{prop:green} on the convergence and quantitative positivity of the Green functions. We follow the classical textbook of Gilbarg and Trudinger \cite[pp. 25-27]{gilbarg2001elliptic}.

We begin by proving the convergence to zero. Fix $\eta,\ep>0$. $\Gamma$ is uniformly continuous away from zero. Therefore, there exists $0 <\delta<\eta$ such that
	\[ |\Gamma(x-y)-\Gamma(x-y')| \le \ep/3 \]
for all $x \in \X_\eta'$ and $y,y' \in \X_{\eta/2}$ with $|y-y'| \le \delta$. By (\ref{eq:barrier-pos}) pick $M$ (depending on $\eta$, $\delta$, and $\mathcal \X$) large enough that for all $p \in \partial \X$,
	\[ M \inf_{|y-p|>\delta}\omega_p^Z(y) \ge 2\sup_{x \in \X_\eta'} \sup_{y' \in \X_{\eta/2}} |\Gamma(x-y')|\]
 Indeed, the right-hand side is an increasing function of $\eta^{-1}$ due to the radially decreasing nature of $\Gamma$. On the other hand, the infimum is a positive number depending only on $\delta$ and the regular boundary family $\mathcal Z$. Next,  for each $x \in \X_\eta'$, consider the two functions $h_1(y) =\Gamma(x-p)-\ep/3-M \omega^Z_p(y)$ and $h_2(y) = \Gamma(x-p)+\ep/3+M\omega_p^Z(y)$. On the boundary $h_1 \le H_\X(x,\cdot) \le h_2$. Indeed, if $|p-q| \le \delta$ then
	\[ |H_\X(x,q)-\Gamma(x-p)|=|\Gamma(x-q)-\Gamma(x-p)| \le \ep/3,\]
and if $|p-q| > \delta$, then by our choice of $M$, $|H_\X(x,q)-\Gamma(x-p)| \le M \omega_p^Z(q)$. Furthermore, $h_1$ is subharmonic and $h_2$ is superharmonic so by the maximum/minimum principles, $h_1 \le H_\X(x,\cdot) \le h_2$ on all of $\X$. Finally, by property (\ref{eq:barrier-zero}) of the regular boundary family, there exists $\delta'>0$ such that $\omega_p^Z(y) \le \ep/(3M)$ if $|y-p| \le \delta'$. Therefore, if $y \in \X_{\min\{\delta,\delta'\}}$,
	\[ \begin{aligned} |H_\X(x,y)-\Gamma(x-y)| &\le |H_\X(x,y)-\Gamma(x-\pi(y))| \\
		&\qquad+ |\Gamma(x-\pi(y)) - \Gamma(x-y)| \\
		&\le \ep/3 + M \omega_{\pi(y)}^Z(y) + \ep/3 \\ &\le \ep. \end{aligned} \]

On the other hand, the positivity of $G_\X$ is controlled by the distance between $H_\X$ and $\Gamma$ near the boundary. Fixing $x \in \X_\eta'$, let $\pi(x)$ be a closest point to $x$ in $\partial \X$. This guarantees $\Gamma(x-\pi(x)) \ge \Gamma(x-p)$ for all $p \in \partial \X$. The line segment $(1-t)x+t\pi(x)$, $t \in (0,1)$ is contained in $\X$ (if it were not contained in $\X$, then $\pi(x)$ would not be a closest boundary point). Define $y_0=\frac{x+\pi(x)}{2}$ so that $d(y_0,\partial \X) = \frac 12 d(x,\partial \X)$. There exists $\ep_0 \in (0,1)$  depending only on $\eta$  such that
	\[ \Gamma(x-\pi(x)) \le \ep_0 \Gamma(x-y_0). \]
Set
	\[ \theta = \frac{(1-\ep_0)\Gamma(x-y_0)}{2M }, \quad M=\sup_{y,p \in \bar \X}\omega_p^Z(y),\]
so that $h_3(y):=\theta \omega_{\pi(x)}^Z(y)+\Gamma(x-\pi(x)) \le \frac{1+\ep_0}{2}\Gamma(x-y_0)$ for all $y \in \X$. Comparing the boundary values, for any $p \in \partial \X$,
	\[ h_3(p) \ge \Gamma(x-\pi(x)) \ge \Gamma(x-p) = H_\X(x,p). \]
Applying the maximum principle again, this inequality extends to the interior and 
	\[ G_\X(x,y_0) = \Gamma(x-y_0)-H_\X(x,y_0) \ge \Gamma(x-y_0)-h_3(y_0) \ge \frac{1-\ep_0}{2}\Gamma(x-y_0). \]
Finally, by the minimum principle and Harnack's inequality, 
	\[ \inf_{y \in \X_\eta'}G_\X(x,y) \ge \inf_{d(y,\partial \X)=\eta/2} G_\X(x,y) \gtrsim G_\X(x,y_0) \gtrsim 1\]
and the implicit constants only depend on $\eta$ and $M$. Most importantly, they are independent of $x$ and $\X$, which proves the proposition.

\section{Application to Pseudoconvex Domains}\label{sec:pseudo}

On the face of it, one could actually decompose a domain $\Omega$ in a variety of ways and apply the results of the previous sections (e.g. a Whitney decomposition yields Theorem \ref{thm:lower-dim-euclid}---see Section \ref{sec:lower-dim}). However, to obtain a characterization of the dominating sets on pseudoconvex domains, we decompose into Kobayashi balls due to their relationship to the Bergman kernel. 

\begin{definition}\label{def:pseudo}
Let $\Omega \subset \mathbb{C}^n$ be a domain with smooth defining funciton $\rho.$ That is, $\Omega= \{z:\rho(z)<0\}$ and $\nabla \rho \neq 0$ on $\partial \Omega.$ We say $\Omega$ is \textit{strongly pseudoconvex} if its defining function $\rho$ is strictly plurisubharmonic.

Recall the notation $\Omega_\ep = \{ z \in \Omega : \delta(z) \le \ep \}$ from (\ref{eq:nbhd}). It is well known that for a smoothly bounded strongly pseudoconvex domain, there exists a unique normal projection to the boundary $\pi: \Omega_\ep \rightarrow \partial \Omega$.  More precisely, there exists $\ep>0$ so that if $z \in \Omega_{\ep}$ , then there is a unique point $\pi(z) \in \partial \Omega$  that minimizes the Euclidean distance of $z$ to the boundary. The map $\pi$ possesses reasonable regularity properties; we refer the reader to Lemma 2.1 in \cite{baloghbonk2000} for a more detailed discussion of the normal projection. 

For a pseudoconvex domain $\Omega$, the \textit{Kobayashi distance} can be defined by its infinitessimal Finsler metric \cite{krantz2001function}, 
	\[ \begin{array}{rcl} F_K(z,\xi) = \inf\{ \alpha>0: &\phi \in \Hol(\mathbb D;\Omega), \\
		&\phi(0)=z, \phi'(0) = \xi/\alpha\}, &(z,\xi) \in \Omega \times \C^n \end{array} \]
	\[ \begin{array}{rcl}
		d_\Omega(z,w) = \inf \{ \int_0^1 F_K(\gamma(t),\gamma'(t)) \, dt : &\gamma \in C^1([0,1];\Omega), \\
		& \gamma(0)=z, \, \gamma(1)=w \} &z,w \in \Omega.\end{array} \]
\end{definition}
$d_\Omega$ is also the smallest distance which is bounded below by the Lempert function,
	\[ \begin{array}{rcl} \inf \{ d_{\mathbb D}(\zeta,\omega), &\phi \in \Hol(\Omega;\mathbb D) \\
			&\phi(z)=\zeta, \, \phi(w)=\omega \}& z,w \in \Omega. \end{array} \]
The Kobayashi distance on the disk, $d_{\mathbb D}$, coincides with the usual Poincar\'e distance. The property concerning the Lempert function actually gives an equivalent definition of $d_\Omega$, see \cite{abate1989iteration,abatesaracco2011}. Since $\Omega$ is fixed throughout, we will just write $d(z,w)$ for $d_\Omega(z,w)$. We will denote Euclidean distance by $|\cdot|$ and the Euclidean distance of the point $z$ to $\partial \Omega$ by $\delta(z)$.  Recall that for any fixed choice of defining function $\rho$, we have the equivalence $|\rho(z)| \approx \delta(z)$ for $z$ in a neighborhood of the boundary.

Denote by $Y(w,r)$ the Kobayashi ball centered at $w$ with radius $\tanh^{-1} r$. 
The local results of Section \ref{sec:local} can be strengthened in the context of Kobayashi balls. We will show that for any $R > r>0$, one can take $Y = Y(w,r)$ and $X=Y(w,R)$ in Proposition \ref{prop:local}.

To do so we will need some properties of the Kobayashi balls which we collect here. Denote by $P(w,r_1,r_2)$ the polydisc
	$$P(w,r_1,r_2):= \{z \in \C^n: |z_1-w_1|<r_1,|z_2-w_2|<r_2,\dots, |z_n-w_n|<r_2\}$$

The first lemma concerns the geometric properties of Kobayashi balls, and will allow us to show that the geometric conditions (i)-(iii) from Proposition \ref{prop:local} are satisfied by $Y=Y(w,r)$ and $X=Y(w,R)$ for any $R>r>0$. 
\begin{lemma}\label{lemma:kob-geo}
Let $\Omega$ a be smoothly bounded strongly pseudoconvex domain. There exists $\ep>0$ and functions $a,b,A,B:(0,1) \to (0,\infty)$ such that for all $w \in \Omega_{\ep}$ and $r>0$, 
	$$w+U(w)^{-1} P(0,a(r) \delta(w),b(r)\delta(w)^{1/2}) \subset Y(w,r) \subset w+ U(w)^{-1} P(0,A(r) \delta(w), B(r) \delta(w)^{1/2})  $$
where $U(w)$ is  any  rotation on $n$ complex variables which rotates the complex normal vector at $\pi(w)$ so it lies in the plane $\C \times \{0\} \times \cdots \times \{0\}$. As a consequence,  there exists a function $C:(0,1) \to (0,\infty)$ satisfying 
	\[ C(r)^{-1} \delta(w)^{n+1} \le |Y(w,r)| \le C(r) \delta(w)^{n+1}. \]
\begin{proof}
Choose $\ep>0$ sufficiently small so that for $w \in \Omega_\ep,$ the normal projection $\pi(w)$ is well-defined. This polydisc containment can be found in the references \cite{krantz1988bloch,li1992}. The volume estimate then easily follows. 
\end{proof}
\end{lemma}

We have another useful lemma concerning the Kobayashi metric and distance to the boundary. 

\begin{lemma}\label{lemma:kobbounddist}
Let $\Omega$ a be smoothly bounded strongly pseudoconvex domain. There exists a function $D:(0,1) \to (0,\infty)$ so that if $d(z,w) \le r$, then 
	\[D(r)^{-1} \delta(w) \le \delta(z) \le D(r) \delta(w). \]
\end{lemma}
\begin{proof}
This statement is proved in \cite{abatesaracco2011}.
\end{proof}

The other ingredient will allow us to extend from $p=2$ to all values of $1 \le p \le \infty$ and obtain the desired form of the constant, polynomial in $\gamma^{-1}$, in all dimensions, not just $n=1$.

\begin{lemma}\label{lemma:kob-mean}
Let $\Omega$ be a  smoothly bounded  strongly pseudoconvex domain and $R>r>0$. There exists $C>0$ such that for all $1\le p,q \le \infty$, $z \in \Omega$, and $f$ which are holomorphic on $Y(z,R)$,
	\[ \avg{f}_{p,Y(z,r)} \le C\avg{f}_{q,Y(z,R)}. \]
\end{lemma}

\begin{proof}
First we use the mean value property of the Kobayashi balls from \cite[Cor. 1.7]{abatesaracco2011}: For each $s>0$, there exists $C>0$ such that
	\[ |f(w)| \le C \avg{f}_{1,Y(w,s)} \]
for all $w \in \Omega$. For $R>r>0$, the triangle inequality shows that $Y(w,R-r) \subset Y(z,R)$ for all $w \in Y(z,r)$. Applying the mean value property with $s=R-r$ and H\"older's inequality,
	\[ \avg{f}_{p,Y(z,r)} \le C \sup_{w \in Y(z,r)} \avg{f}_{q,Y(w,R-r)} \le \sup_{w \in Y(z,r)} \frac{C \|f\|_{L^q(Y(z,R))}}{|Y(w,R-r)|^{1/q}} . \]
However, $|Y(w,R-r)|$ and $|Y(z,R)|$ are only off by a constant depending on $R$ and $r$ by Lemma \ref{lemma:kob-geo}. 
\end{proof}

Now we are ready to prove our first local estimate for Kobayashi balls. The lower dimensional one will be proved later, in Section \ref{sec:lower-dim}. We introduce the doubling index
	\begin{equation}\label{eq:N} N_p(f,z,r,R) = \log\frac{\|f\|_{L^p(Y(z,R))}}{\|f\|_{L^p(Y(z,r))}}. \end{equation}
When any of the parameters $p,f,z,r,R$ are apparent, we will drop them from the notation. 
\begin{lemma}\label{lemma:local}
Let $\Omega \subset \C^n$ be a smoothly bounded strongly pseudoconvex domain, $R>r>0$, and $1 \le p \le \infty$. There exists $C,\ep_0>0$ such that  for all $\gamma>0,$  $z \in \Omega_{\ep_0}$ and $E \subset Y(z,r)$ satisfying $\avg{E}_{Y(z,r)} \ge \gamma$,
	\begin{equation}\label{eq:remez} \|f\|_{L^p(Y(z,r))} \le \left( \frac{C}{\gamma}  \right)^{ C(N+1) } \|f\|_{L^p(E)}, \quad N=N_p(f,z,r,R),\end{equation}
for all $f \in A^p(Y(z,R))$.
\end{lemma}

\begin{proof}
 We aim to use Proposition \ref{prop:local} . The natural affine map to take is $D = \Lambda_\delta(w) \circ U(w)$ where $\Lambda_\delta(w)$ is defined by the scaling $(z_1,z_2,\ldots,z_n) \mapsto \lambda( z_1,\delta(w)^{1/2}z_2,\ldots,\delta(w)^{1/2}z_n)$  and $U$ is any unitary as in Lemma \ref{lemma:kob-geo}; a particular $U(w)$ will be chosen later. In this way, choosing $\lambda$ so that $\diam D[Y(w,r)]=\frac 14$,
	\[ \begin{aligned} 4^{-2n} &= \diam D[Y(w,r)]^{2n} \le [\lambda\delta(w)]^{2n}[(n-1)B(r)^2+A(r)^2]^{n}, \\ 
			|D[Y(w,r)]| &\ge (\lambda\delta(w))^{2n}b(r)^{2n-2}a(r)^2, \end{aligned}\]
 which provides a lower bound on $|D[Y(w,r)]|$ that only depends on $r.$ 
			
It remains to establish the boundary separation property.  Choose $\ep_0$ sufficiently small so that if $w \in \Omega_{\ep_0},$ then $Y(w,R) \subset \Omega_\ep,$ where $\ep$ is as in Lemma \ref{lemma:kob-geo}. 
The triangle inequality shows that one can fit $Y(v,(R-r)/2) \subset Y(w,R) \backslash Y(w,r)$ as long as $d(v,w) = (R+r)/2$. Therefore, on applying $D$,
	\begin{equation}\label{eq:sep} \begin{aligned} \dist(\partial D(Y&(w,R)),\partial D(Y(w,r)) \\
		&\ge \inf_{d(v,w)=\tfrac{R+r}2} \inf \{2|x-Dv| : x \in \partial D(Y(v,\tfrac{R-r}{2})) \} \\
		&\ge \inf_{d(v,w)=\tfrac{R+r}2} \inf \{ 2|\Lambda_\delta(w) U(w) U(v)^{-1} x| : \\
		& \qquad \qquad \qquad x \in \partial P(0,\delta(v)a(\tfrac{R-r}{2}),\delta(v)^{1/2}b(\tfrac{R-r}{2}))\} \end{aligned} \end{equation}
Since $\Omega$ is smooth, for every $\eta>0$, there exists $\ep'>0 $ such that if $p,q \in \bomega$ with $|p-q| \le  \ep' $, then 
	\[ |n(p)-n(q)| \le \eta|p-q|, \]
where $n(p)$ is the outward normal vector at $p \in \partial \Omega$. For $w \in \Omega_{\ep_0}$, let $n(w)=n(\pi(w))$.
Now we specify $U(w)$ and $U(v)$ to be unitary coordinate changes satisfying
	\[ |U(w)x-U(v)x| \le C |x| \cdot |n(w)-n(v)|, \quad x \in \mathbb C^n\]
where $C$ is some absolute constant (we are thinking of $v$ as belonging to a suitable small neighborhood of $w$ with respect to the Euclidean distance).
Therefore, by unitarity 
	\[ |U(w)U(v)^{-1}x - x| \le C | x||n(w)-n(v)|, \]
and if $|\pi(w)-\pi(v)| \le \ep'$  and  we assume $\delta(w) \leq 1$, then
	\[ |\Lambda_\delta U(w)U(v)^{-1} x - \Lambda_\delta x| \le \eta \lambda C |\pi(w)-\pi(v)| \cdot |x|. \]
So, if $d(w,v) \le R$, then by using the triangle inequality and Lemmas \ref{lemma:kob-geo} and \ref{lemma:kobbounddist}, we have 
$$
|\pi(w)-\pi(v)| \le |\pi(w)-w|+|w-v|+|v-\pi(v)| \lesssim (1+D(R)+\max\{A(R),B(R)\})\delta(w)^{1/2}.
$$

 Therefore, perhaps by shrinking $\varepsilon_0$, if we restrict to $\delta(w)^{1/2} \le  \min\{1,\ep' (1+D(R)+B(R))^{-1}\}$,  combining the previous two displays gives a uniform lower bound for (\ref{eq:sep}). Indeed, for any $w \in \Omega_{\ep_0}$, $v$ such that $d(v,w) = \frac{R+r}{2}$, and $x \in \partial P(0,\delta(v)a(\frac{R-r}{2}),\delta(v)^{1/2} b(\frac{R-r}{2}))$, 
	\[\begin{aligned} |\Lambda_{\delta} U(w)U(v)^{-1}x| &\ge |\Lambda_\delta x| - \eta \lambda C |\pi(w)-\pi(v)| \cdot |x| \\
	&\ge |\Lambda_\delta x| - C \eta \lambda \delta(w)^{1/2} |x| \\
	&\ge |\Lambda_\delta x| - C \eta |\Lambda_\delta x| \\
	&\gtrsim \lambda \delta(w)(1-C \eta). \end{aligned}\]
But $\eta$ can be chosen arbitrarily small . 
Thus the boundary separation property follows from the normalizing choice of $\lambda$ at the beginning.

 Fix $s$ and $S$ satisfying $r<s<S<R$. Applying Proposition \ref{prop:local} to $Y(z,s) \subset Y(z,S)$, we obtain
	\begin{equation}\label{eq:lpl2} \avg{f}_{p,Y(z,r)} \le C \avg{f}_{2,Y(z,s)} \le Ce^{C(N_2(s,S)+1)S^*} \avg{f}_{2,E}  \end{equation}
for any  measurable  $E \subset Y(z,r) \subset Y(z,s)$  with $|E|>0$. However,
	\[ N_2(s,S) = \log \frac{\|f\|_{L^2(Y(z,S))}}{\|f\|_{L^2(Y(z,s))}} \le C+ \log \frac{\|f\|_{L^\infty(Y(z,S))}}{\|f\|_{L^\infty(Y(z,r))}}=C+ N_\infty(r,S). \]
Therefore,
	\begin{equation}\label{eq:ndim} \sup_{Y(z,r)} |f| \le Ce^{C(N_\infty(r,S)+1)S^*} \sup_{E}|f|. \end{equation}
We can now address the form of the constant in  the statement  for all $n \ge 1$ using the following rotation argument which is common in results of this type \cite{kovrijkine01,lebeau19}.  If one does not care about the sharp form of the constant, one may skip to (\ref{eq:G}) and use the less precise estimate (\ref{eq:ndim}) in place of (\ref{eq:1d}) below.

According to Lemma \ref{lemma:kob-geo}, the rotated and translated polydisk 
	\[ P^*(z) := z+U(z)^{-1}P(0,\tfrac{a(r)}{2} \delta(z),\tfrac{b(r)}{2} \delta(z)^{1/2}) \]
is strictly contained in $Y(z,r)$ and the density $\avg{P^*(z)}_{Y(z,r)}$ is bounded below by some postive constant depending on $r$. Furthermore, there exists $z_0 \in P^*(z) \subset Y(z,r)$ such that $\sup_{P^*(z)}|f| \le 2 |f(z_0)|$. Therefore, applying (\ref{eq:ndim}) with $E=P^*(z)$, there exists a constant $C$ depending on $r$ such that
	\begin{equation}\label{eq:z0} \sup_{Y(z,r)} |f| \le Ce^{C(N_\infty(r,S)+1)}|f(z_0)|.\end{equation}
At this point, we rescale by the affine map $D=D(z)$ from the beginning of this proof. Define $g = f(D^{-1}\cdot+z_0)$, $F = D(E-z_0)$, $Y = D(Y(z,r)-z_0)$, and $X = D(Y(z,S)-z_0)$. Let $\kappa$ denote an affine complex line through the origin; that is, $\kappa$ is determined by a certain direction $\Theta \in \mathbb C^n$ with $|\Theta|=1$ so that
	\[ \kappa = \{z \in \mathbb C^n : z=w\Theta,\ \mbox{for some} \ w \in \mathbb C \}.\] Then, for a measurable set $A \subset \mathbb C^n$, the section $A_\kappa$ is defined by $\{ w \in \mathbb C : z=w\Theta \in A \}$. We claim now that there exists a constant $a$, depending on $\Omega$ and $r$, and an affine complex line through the origin $\kappa$ such that 
	\begin{equation} \label{eq:kappa} |F_\kappa| \ge a \gamma \cdot |Y_\kappa|.\end{equation} 
Note that any arbitrary section $A_\kappa$ has real dimension two so $|\cdot|$ in (\ref{eq:kappa}) means the two dimensional Lebesgue measure. We will return to establishing (\ref{eq:kappa}) at the end of the proof.

We have already verified the geometric conditions (i)-(iii) in Proposition \ref{prop:local} for these $X$ and $Y$. Let us confirm they are preserved under taking sections, $X_\kappa$, $Y_\kappa$, up to isotropic dilation. First, 
	\[ \operatorname{dist}(\partial X_\kappa,\partial Y_\kappa) = \inf_{\substack{ w_1\Theta \in \partial X \\ w_2\Theta \in \partial Y}} |w_1-w_2| = \inf_{\substack{ w_1\Theta \in \partial X \\ w_2\Theta \in \partial Y}} |w_1 \Theta -w_2 \Theta| \]
	\[ \ge \inf_{\substack{ z_1 \in \partial X \\ z_2 \in \partial Y}} |z_1 -z_2| = \operatorname{dist}(\partial X,\partial Y) \ge d.\] \
Next, since $z_0 \in P^*(z)$, the triangle inequality shows that $Y$ contains a Euclidean ball centered at the origin of some small radius $c_r>0$. Similar to the above computation, crucially using the fact $|\Theta|=1$, we obtain $|B(0,c_r)_\kappa| = \pi c_r^2$ and $\diam(B(0,c_r)_\kappa)=2c_r$. By containment, these gives lower bounds on $|Y_\kappa|$ and $\diam(Y_\kappa)$. 
Therefore, to apply Proposition \ref{prop:local} we only need to apply the isotropic dilation by some parameter $b>0$ so that $b \cdot \diam(Y_\kappa) = \frac 14$. Since $b$ is bounded above and below by constants depending on $r$, Proposition \ref{prop:local} also applies to $X_\kappa$, $Y_\kappa$.

Now, $h(w) := g(w\Theta)$ for $w \in \mathbb C$ is a holomorphic function of one complex variable. The mean value property is immediate in the plane, so, by (\ref{eq:z0}), (\ref{eq:kappa}), and the $n=1$ case of Proposition \ref{prop:local}, 
	\begin{equation}\begin{aligned} \label{eq:1d}  Ce^{-C(N_\infty(r,S)+1)} \sup_{Y(z,r)}|f| &\le |f(z_0)| \le \sup_{Y_\kappa} |h| \\
	& \le \left(\frac{C}{\gamma}\right)^{C(\tilde N+1)} \sup_{F_\kappa} |h|  \le \left(\frac{C}{ \gamma}\right)^{C(N_\infty(r,S)+1)} \sup_{E}|f|, \end{aligned} \end{equation}
where, by (\ref{eq:z0}),
	\[ \begin{aligned} \tilde N &= \log \frac{\sup_{X_\kappa}|h|}{\sup_{Y_\kappa}|h|} \le \log \frac{\sup_{Y(z,S)}|f|}{|f(z_0)|} \\
	& \le \log \frac{ \sup_{Y(z,S)}|f|}{Ce^{-C(N_\infty(r,S)+1)} \sup_{Y(z,r)}|f|} \le C(N_\infty(r,S)+1). \end{aligned} \]

Finally, we extend (\ref{eq:1d}) to all $1 \le p < \infty$ which will conclude the proof. Define
	\begin{equation}\label{eq:G} G = \{ z \in E : |f(z)| \le 2^{1/p} \avg{f}_{p,E} \}. \end{equation}
Then, $|E \backslash G| \le |E|/2 $ so $\avg{G}_{Y(w,r)} \ge \gamma/2$. Therefore, applying (\ref{eq:1d}) with $G$ in place of $E$, one obtains
	\[ \avg{f}_{p,Y(z,r)} \le \sup_{Y(z,r)} |f| \le \left(\frac{2C}{ \gamma}\right)^{C(N_\infty(r,S)+1)} \sup_{G}|f| \le \left(\frac{C}{ \gamma}\right)^{C(N_p(r,R)+1)}2^{1/p}\avg{f}_{p,E}. \]

This completes the proof, except for constructing the complex line $\kappa$ so that (\ref{eq:kappa}) holds. We do so using complex polar coordinates which can be introduced using projective geometry, but we prefer to do this by hand as a calculus exercise. For each $k=1,\ldots,n$, define 
	\[ \mathcal C^n_k = \{z \in \mathbb C^n : |z_i| < \frac{1}{\sqrt{n}}|z|, \ i=1,\ldots,k-1, \ |z_k | \ge \frac{1}{\sqrt n}|z| \}.\]
For any $z \in \mathbb C^n \backslash \{0\}$ there exists a smallest $k$ such that $|z_k| \ge \frac{1}{\sqrt n} |z|$. Then $z$ belongs to this $\mathcal C^n_k$ and none other. Therefore $\{\mathcal C^n_k\}_{k=1}^n$ forms a partition of $\mathbb C^n \backslash \{0\}$. Now we define the polar coordinates for $z $ in each $\mathcal C^n_k$:
	\[ w = \frac{z_k}{|z_k|} |z| \in \mathbb C , \quad \rho = \frac{|z_k|}{z_k} \frac{(z_1,\ldots,z_{k-1},z_{k+1},\ldots,z_n)}{|z|} \in \mathcal B^{n-1}_k, \]
	\[ \mathcal B^{n-1}_k := \{ z \in\mathbb C^{n-1} : |z_i| < \tfrac{1}{\sqrt n}  \ i=1,\ldots,k-1, \ |z| \leq \sqrt{1-\tfrac{1}{n}} \}.\]
For ease of notation, we introduce the variable 
	\[ \Theta = \Theta(\rho) = (\rho_1,\ldots,\rho_{k-1},\sqrt{1-|\rho|^2},\rho_{k+1},\ldots,\rho_n) \in \mathbb C^n, \quad |\Theta|=1\]
so that $z=w\Theta$. We think of $\Theta$ as the direction of the complex line over which $z$ varies as $w$ varies over $\mathbb C$.

Some calculations show that the Jacobian of this tranformation has determinant $|w|^{2n-2} q(\rho)$ where $q$ is some rational function which blows up as $|\rho| \to 1$. However, this was the reason we constructed each $\mathcal C^n_k$ because in that case $|\rho|^2 = \frac{|z|^2-|z_k|^2}{|z|^2} \leq 1-\frac{1}{n}$ is bounded away from $1$ and so $q$ is bounded on each piece $\mathcal B^{n-1}_k$. 
Therefore, we have the polar coordinate transformation for any $H \in L^1(\mathbb C^n)$,
	\[ \begin{aligned} \int_{\mathbb C^n} H(z) \, dA(z) &= \int_{\mathbb C^n \backslash \{0\}} H(z) \, dA(z) = \sum_{k=1}^n \int_{\mathcal C^n_k} H(z) \, dA(z) \\
	&= \sum_{k=1}^n \int_{\mathcal B^{n-1}_k} \int_{\mathbb C} H(w \Theta(\rho)) |w|^{2n-2} |q(\rho)| \, dA(w) \, dA(\rho).\end{aligned}\]
Suppose, toward a contradiction, that 
	\begin{equation}\label{eq:contra} \int_{\mathbb C} 1_F(w\Theta(\rho)) \, dw \le a \gamma \cdot \int_{\mathbb C} 1_Y(w\Theta(\rho)) \, dw \end{equation}
for some $a>0$ and all $\rho \in \mathcal B := \cup_{k=1}^n \mathcal B^{n-1}_k$. By Lemma \ref{lemma:kob-geo}, $Y$ is contained in a ball centered at the origin, say $B_r$ whose radius depends only on $r$, the radius of the original Kobayshi ball $Y(z,r)$. Furthermore, $|Y|$ is bounded below by some constant depending only on $\Omega$, $r$, and $n$. With these remarks, the assumption (\ref{eq:contra}), and the polar coordinates, we obtain
	\[ \begin{aligned} \gamma |Y| &\le \int_{\mathbb C^n } 1_F dA =\sum_{k=1}^n \int_{\mathcal B^{n-1}_k } \int_{\mathbb C} 1_F(w\Theta(\rho)) |w|^{2n-2} \, dw |q(\rho)| \, d\rho \\
	& \le \diam(B_r)^{2n-2} \sum_{k=1}^n \int_{\mathcal B_k^{n-1}} \int_{\mathbb C} 1_F(w\Theta(\rho)) \, dw |q(\rho)| \, d\rho  \\
	& \le a \gamma \diam(B_r)^{2n-2} \sum_{k=1}^n \int_{\mathcal B^{n-1}_k} \int_{\mathbb C} 1_Y(w\Theta(\rho)) \, dw |q(\rho)| \, d\rho\\
	& \le a \gamma \diam(B_r)^{2n} \left( \sum_{k=1}^n \int_{\mathcal B^{n-1}_k} |q(\rho)| \, d\rho \right). \end{aligned} \]
Cancelling out $\gamma$, the constants above only depend on $r$ and $n$ so for $a$ small enough, the above display is a contradiction. Therefore there must be a $\Theta^*$ so that the opposite inequality to (\ref{eq:contra}) holds. This is exactly the desired estimate (\ref{eq:kappa}) where
	\[ \kappa = \{ z \in \mathbb C^n : z= w \Theta^*, \ w \in \mathbb C \}. \]

\end{proof}

It will also be important later that Lemma \ref{lemma:local} holds when $Y(z,R)$ and $Y(z,r)$ are replaced by any Euclidean balls $B(z,R)$ and $B(z,r)$ or Euclidean cubes $Q(z,R)$ and $Q(z,r)$ with a fixed ratio $\frac{R}{r}=d$. This can be proved following the same path as above, yet with many simplifications since the precise geometry is known. For this reason, and because this result follows from a theorem of A. Brudnyi in \cite{brudnyi1999local}, we omit the proof.

\subsection{Sufficiency}\label{subsec:suff}
We are now in the position to prove ``(ii) implies (i)'' in Theorem \ref{thm:main}.
\begin{proposition}\label{prop:suff}
Let $\Omega$ be a smoothly bounded strongly pseudoconvex domain, $0<r<1$, $1 \le p \le \infty$, $\alpha>-1$. There exists $C,q>0$ such that for all $0<\gamma<1$ and $E \subset \Omega$ satisfying
	\begin{equation}\label{eq:density-suff} \frac{|E \cap Y(z,r)|}{|Y(z,r)|} \ge \gamma, \quad \forall z \in \Omega,\end{equation}
	\[ \|f\|_{L^p_\alpha(\Omega)} \le C \gamma^{-q} \|f\|_{L^p_\alpha(E)} , \quad \forall f \in A^p_\alpha(\Omega).\]
\end{proposition}
\begin{proof}
Fix $R>r$, $1 \le p \le \infty$, and $\ep$ from Lemma \ref{lemma:local}. We first claim that $\Omega_\ep$ is a dominating set. If not, then we can find $\{f_n\} \subset A_\alpha^p(\Omega)$, with $\|f_n\|_{L^p_\alpha(\Omega)}=1$ and $\|f_n\|_{L^p_\alpha(\Omega_\ep)} \to 0$. The second property implies the existence of subsequence $\{f_{n_k}\}$ such that 
	\begin{equation}\label{eq:zero-ep} \lim_{k \to \infty} f_{n_k}(z) = 0 , \quad z \in \Omega_\ep \mbox{ a.e.}\end{equation}
On the other hand, since point evaluation is a bounded linear functional on $L^p_\alpha(\Omega)$, $\{f_{n_k}\}$ is a normal family. Consider the compact set $K=\overline{ \Omega \backslash \Omega_{\ep/2}}$. There exists a subsequence $f_{m_k} \subset \{f_{n_k}\}$ and $f_0 \in \Hol(\Omega)$  such that
	\[ \lim_{k \to \infty} f_{m_k}(z) = f_0(z), \quad z \in K. \]
By (\ref{eq:zero-ep}), $f_0(z)=0$ for $z \in K \cap \Omega_\ep$ which implies $f_0(z)=0$ for all $z \in K$. However, since $K \cup \Omega_\ep = \Omega$,
	\[ 0=\|f_0\|_{L^p_\alpha(K)} = \lim_{k \to \infty} \|f_{m_k}\|_{L^p_\alpha(K)} \ge \lim_{k \to \infty} \|f_{m_k}\|_{L^p_\alpha(\Omega)} - \|f_{m_k}\|_{L^p_\alpha(\Omega_\ep)} =1 \]
which is a contradiction. Therefore, there exists $C_\ep>0$ such that
	\[ \|f\|_{L^p_\alpha(\Omega)} \le C_\ep \|f\|_{L^p_\alpha(\Omega_\ep)}. \]
Now, $\Omega$ can be covered with Kobayashi balls $Y_k=Y(w_k,r)$ such that $X_k=Y(w_k,R)$ have finite overlap of say $M>0$, see \cite[Lemma 1.5]{abatesaracco2011}. For $1 \le p <\infty$, call an index $k$ good if 
	\[ (2C_\ep M)^{1/p} \|f\|_{L^p_\alpha(Y_k)} \ge \|f\|_{L^p_\alpha(X_k)}. \]
Define 
	\[  \mathcal G = \{ k : k \mbox{ is good and } w_k \in \Omega_\ep \}, \quad
		\mathcal B = \{ k : w_k \in \Omega_\ep \} \backslash \mathcal G. \]
In this way
	\[ \sum_{k \in \mathcal B} \|f\|_{L^p_\alpha(Y_k)}^p \le \dfrac{1}{2C_\ep M} \sum_{k \in \mathcal B} \|f\|_{L^p_\alpha(X_k)}^p \le \frac 1{2C_\ep} \|f\|_{L^p_\alpha(\Omega)}^p  \le \frac 12 \|f\|_{L^p_\alpha(\Omega_\ep)},\]
so $\sum_{k \in \mathcal G} \|f\|_{L^p_\alpha(Y_k)}^p \ge \frac 12 \|f\|^p_{L^p_\alpha(\Omega_\ep)}$. Thus, for each \textit{good} $k$, we apply Lemma \ref{lemma:local} and the last statement of Lemma \ref{lemma:kobbounddist} to obtain
	\[ \begin{aligned}\|f\|_{L^p_\alpha(Y_k)} &\le D(r)|\rho(w_k)|^{\alpha} \| f \|_{L^p(Y_k)} \\
			&\le D(r)|\rho(w_k)|^{\alpha}\left(\frac{C}{\gamma} \right )^{C(N+1) } \|f\|_{L^p(E \cap Y_k)} \\
			&\le D(r)^2 \left(\frac{C}{\gamma} \right )^{C'} \|f\|_{L^p_\alpha(E \cap Y_k)} . \end{aligned}\]
We crucially used the fact that $N_p(w_k,f,r,R)$ is bounded by some constant depending on $D(r)$, $M$ and $C_\ep$ when $k$ is good.
Therefore,
	\[ \|f\|_{L^p_\alpha(\Omega)}^p \le 2C_\ep^p \sum_{k \in \mathcal G} \|f\|_{L^p_\alpha(Y_k)}^p \le C\gamma^{-C'} \|f\|^p_{L^p_\alpha(E)}. \]
When $p=\infty$, the proof is easier, noticing that there exists $z_0 \in \Omega$ such that $\|f\|_\infty \le 2 |f(z_0)|$. Then, $\sup_{Y(z_0,R)}|f| \le 2\sup_{Y(z_0,r)}$ so $N \le \log 2$ and Proposition \ref{prop:local} implies
	\[ \|f\|_\infty \le 2 \sup_{Y(z_0,r)}|f| \le C\gamma^{-C'} \sup_{E} |f|. \]
\end{proof}

\begin{remark}
From the proof, it is clear that the only property of the weight $\rho^\alpha$ that we used was that it is approximately constant on each Kobayashi ball. So Proposition \ref{prop:suff} holds for any $L^p_w$ space with a weight $w$ satisfying this assumption.
\end{remark}

\subsection{Necessity}\label{subsec:nec}
Recall from the introduction (\ref{eq:intro-berezin}) that $T_E^{p,\alpha}(z) = \|k_z^{p,\alpha}\|_{L^p_\alpha(E)}$. An obvious necessary condition for $E$ to be a dominating set for $A^p_\alpha(\Omega)$ is that 
	\[ \inf_{z \in \Omega} T_E^{p,\alpha}(z) >0.\]
Recall also that $k_z^{p,\alpha}$ is the $L^p_\alpha$-normalized reproducing kernel for $A^2_\alpha(\Omega)$. If $\|k_z^{p,\alpha}\|_{L^p_\alpha(E)}$ can be connected to the density condition, then we can find dominating sets simply by testing on certain functions.  In the case that $\alpha$ is a non-negative integer,  the following two properties of the reproducing kernel facilitate this connection.
\begin{lemma}\label{lemma:kernel}
Let $\Omega$ be a smoothly bounded strongly pseudoconvex domain. Let $1 < p < \infty$ and  $\alpha$ be a non-negative integer . Then,
	\begin{equation}\label{eq:tail} \lim_{r \to 1^- } \sup_{z \in \Omega} \|k_z^{p,\alpha}\|_{L^p_\alpha(\Omega \backslash Y(z,r))} = 0\end{equation}
and for each $r \in (0,1)$,
	\begin{equation}\label{eq:linfty} \sup_{\substack{z \in \Omega\\ w \in Y(z,r)}}|Y(z,r)|^{1/p} \cdot |k_z^{p,\alpha}(w)| |\rho(w)|^{\alpha/p} < \infty. \end{equation}
\end{lemma}
We will prove this Lemma in Section \ref{subsec:kernel}. Assuming it for now, we can complete the proof of Theorem \ref{thm:main}, characterizing the dominating sets for $A^p_\alpha(\Omega)$.
\begin{proof}[Proof of Theorem \ref{thm:main}]
Proposition \ref{prop:suff} from the last section proves (ii) implies (i). (i) implies (iii) is obvious so it remains to prove (iii) implies (ii) using Lemma \ref{lemma:kernel}. Let $c=\inf_{z \in \Omega} \|k_z^{p,\alpha}\|_{L^{p}_{\alpha}(E)}$, and
 $$C_r=\sup_{\substack{z \in \Omega\\ w \in Y(z,r)}}|Y(z,r)|^{1/p} \cdot |k_z^{p,\alpha}(w)| |\rho(w)|^{\alpha/p} .$$
  Pick $r$ such that
	\[ \sup_{z \in \Omega}\|k_z^{p,\alpha}(w)\|_{L^p_\alpha (\Omega \backslash Y(z,r))} < \frac {c} {2^{1/p}} \]
and then
	\[ c^p \le \|k_z^{p,\alpha}\|_{L^p_\alpha(E \cap Y(z,r))}^p + \|k_z^{p,\alpha}\|_{L^p_\alpha(\Omega \backslash Y(z,r))}^p \le C_r^p \frac{|E \cap Y(z,r)|}{|Y(z,r)|} +  \frac {c^p} {2}\]
which shows that $E$ is relatively dense.
\end{proof}

\subsection{Bergman kernel}\label{subsec:kernel}
Is this section, we prove Lemma \ref{lemma:kernel}, establishing estimates (\ref{eq:tail}) and (\ref{eq:linfty}) on the Bergman kernel. Concerning the $L^\infty$ estimate (\ref{eq:linfty}), we first establish it for $z$ away from the boundary. Consider the region 
	\[  U_\ep  = \{ (z,w) \in \bar \Omega \times \bar \Omega : |\rho(z)| + |\rho(w)| + |z-w|^2 \le \ep \}. \]
It is a well-known fact that the function $(z,w) \mapsto k_z(w)$ extends to a  $C^\infty$ map on the closed set $(\bar \Omega \times \bar \Omega) \setminus U_{\ep}$ (meaning all orders of derivatives of this map extend to continuous mappings on $(\bar \Omega \times \bar \Omega) \setminus U_\ep$) \cite{kerzman1972} . Therefore it is bounded. Moreover, $|Y(z,r)| \sim \delta(z)^{n+1}$ which shows for any $r,\ep > 0$ using the fact that $\alpha \geq 0$ and $|\rho(z)| \sim |\rho(w)|$ for $w \in Y(z,r),$

	\[ \sup_{\substack{z,w \in (\Omega \times \Omega) \backslash U_\ep \\ w \in Y(z,r)} } |Y(z,r)|^{1/p} |k_z^{p,\alpha}(w)| |\rho(w)|^{\alpha/p} \le C_{r,\ep}.\] 
On the other hand, pick $\ep$ so small that we have the asymptotic expansion for $(z,w) \in U_\ep$ (see \cite{peloso1994}):
	\[ K_\alpha(z,w) = a(w)|\Psi(z,w)|^{-n-1-\alpha} + E(z,w) ,\]
where $a$ is a non-vanishing smooth function, $\Psi(z,w)$ is a smooth function that is a perturbation of the Levi polynomial, and $E(z,w)$ satisfies $|E(z,w)| \lesssim |\Psi(z,w)|^{-n-1-\alpha+{1/2}} |\log |\Psi(z,w)||.$ In particular, $\Psi$ satisfies 

$$|\Psi(z,w)| \sim |\rho(z)|+|\rho(w)|+|z-w|^2+|\text{Im}\Psi(z,w)|.$$
Then, note that for these $z,w$ there holds $$|K_\alpha(z,w)| \sim  |\Psi(z,w)|^{-n-1-\alpha}\lesssim \delta(z)^{-n-1-\alpha} .$$
Moreover, a computation using the Rudin-Forelli estimates given in \cite{xiawang2021} shows that  provided $p>1$ , 
$$\|K_{\alpha}(z,\cdot)\|_{L^p_\alpha(\Omega)} \sim \delta(z)^{-(n+1+\alpha)/p'}.$$
Finally, we estimate
\begin{align*}
 & \sup_{\substack{z,w \in U_\ep \\ w \in Y(z,r)} } |Y(z,r)|^{1/p} |k_z^{p,\alpha}(w)| |\rho(w)|^{\alpha/p} \\
&  \le C_{r,\ep} \delta(z)^{(n+1)/p} \delta(z)^{(n+1+\alpha)/p'} \delta(z)^{-n-1-\alpha} \delta(z)^{\alpha/p}\\ 
& = C_{r,\ep}, 
\end{align*}
which completes the proof of (\ref{eq:linfty}).

Next, we aim to show (\ref{eq:tail}) holds.
We need the following Lemma concerning the behavior of the Kobayashi metric. Let $z=(z_1,z')$ and $w=(w_1,w')$ be the splitting of $z$ and $w$ into the complex normal and tangential directions based at $\pi(z)$, where $\pi$ denotes the normal projection to the boundary (see \cite{baloghbonk2000}). We will assume $\delta(z)<\varepsilon_0$ so that the properties of Lemma $2.1$ in \cite{baloghbonk2000} are satisfied. Then $z_1,w_1 \in \mathbb{C}$ and $z',w' \in \mathbb{C}^{n-1}$. 

\begin{lemma}\label{kobayashishape} 
Let $\Omega$ be a strongly pseudoconvex domain with $C^2$ boundary. The Kobayashi metric $d$ has the following property: There exists $\ep_0>0$ such that for any $N \in \mathbb{N}$, there exists a radius $R_N \in (0,1)$ such that if $\delta(z) \le \ep_0$ and $d(z,w) \geq \tanh^{-1} R_N$, then

$$|z'-w'|^2+ |z_1-w_1| \geq  N \delta(z) \text{ or } \delta(w)< \frac{1}{N} \delta(z).$$
\end{lemma}

To prove this, we will use the following bounds derived by Balogh and Bonk for the Kobayashi metric \cite[Cor. 1.3]{baloghbonk2000}. There exists $C>0$ such that

$$g(z,w)-C \leq d(z,w) \leq g(z,w)+C$$

where $$g(z,w):= 2 \log{\left[ \frac{d_{H}(\pi(z),\pi(w))+h(z) \vee h(w)}{\sqrt{h(z) h(w)}}  \right]}.$$

Here $\pi(z)$ denotes the normal projection of $z$ to the boundary, $d_H$ denotes the Carnot-Carath\`{e}odory metric on $\partial \Omega$ and $h(z) \vee h(w)= \max\{h(z),h(w)\},$ where $h(z)=\delta(z)^{1/2}$. The Carnot-Carath\`eodory metric $d_H$, also called the horizontal metric, is defined by
	\[ \begin{array}{rcl}
		d_H(p,q) = \inf \{ \int_0^1 L_\rho(\gamma(t),\gamma'(t)) \, dt : &\gamma \in C^1([0,1];\partial\Omega), \\
\gamma'(t) \in H_{\gamma(t)} \partial \Omega		& \gamma(0)=p, \, \gamma(1)=q \}& p,q \in \partial \Omega.\end{array} \]
Here $L_\rho$ denotes the Levi form, and  $H_{\gamma(t)}$ the ``horizontal" or complex tangential subspace at the boundary point $\gamma(t).$ The important estimate for us concerning $d_H$ is the Box-Ball Estimate \cite[Prop. 3.1]{baloghbonk2000} which states that there exist $C,\ep_0>$ such that
	\begin{equation}\label{eq:box-ball} \operatorname{Box}(p,\ep/C) \subset B_H(p,\ep) \subset \operatorname{Box}(p,C\ep) \end{equation}
where $B_H(p,\ep) = \{q \in \partial \Omega : d_H(p,q) \le \ep \}$ and $\operatorname{Box}(p,\ep) = \{ q \in \partial \Omega : |(p-q)_1| \le \ep^2, \, |(p-q)'| \le \ep \}$ where the spltting $z=(z_1,z')$ into complex normal and tangential directions is done at $p$.

We remark that if $w$ is far from the boundary, the projection $\pi$ may not be uniquely defined, but this does not cause problems. In this case, we simply choose $\pi(w)$ to be a point on the boundary satisfying $|\pi(w)-w|=\delta(w).$ In this way, we can extend $\pi$ to a (non-unique) map $\Omega \rightarrow \partial \Omega$ that will satisfy the above estimate.

\begin{proof}[Proof of Lemma \ref{kobayashishape}]
Suppose $d(z,w) \geq R$. Then $g(z,w) \geq R'$, where $R'=R-C.$ Equivalently,

$$\left[ \frac{d_{H}(\pi(z),\pi(w))+h(z) \vee h(w)}{\sqrt{h(z) h(w)}}  \right]\geq R'',$$ where $R''=\exp{(R'/2)}.$ This implies either $\frac{d_{H}(\pi(z),\pi(w))}{\sqrt{h(z) h(w)}} \geq \frac{R''}{2}$ (Case $1$) or $\frac{h(z) \vee h(w)}{\sqrt{h(z) h(w)}} \geq \frac{R''}{2}$ (Case $2$). 

\textbf{Case 1a:}

In Case $1$, we first consider the further subcase where $\delta(w) > \frac{1}{R''} \delta(z).$ Note that the condition for Case $1$ implies

$$d_H(\pi(z),\pi(w)) \geq \frac{R''}{2} \delta(z)^{1/4}\delta(w)^{1/4}$$
and the sub-condition further implies
$$d_H(\pi(z),\pi(w)) \geq \frac{(R'')^{3/4}}{2} \delta(z)^{1/2}.$$

This is a good bound. We may assume without loss of generality that $\frac{(R'')^{3/4}}{2} \delta(z)^{1/2}\leq \varepsilon_0,$ where $\varepsilon_0$ is chosen so that (\ref{eq:box-ball}) holds. If not, we may replace $R$ with a sufficiently large value to force one of the other cases (note that the horizontal metric $d_H$ is a bounded function). 

 Assuming the reduction, write $\pi(z)=(\pi(z)_1,\pi(z)')$ and $\pi(w)=(\pi(w)_1,\pi(w)')$ , where the decomposition of a point $p \in \partial \Omega$ by $p=(p_1,p')$ is given by splitting into the complex normal and tangential directions at $\pi(z).$ Similarly, write $z=(z_1,z')$ and $w=(w_1,w')$ with the splitting based at $\pi(z).$ Then we have, using (\ref{eq:box-ball}):

$$ |\pi(z)_1-\pi(w)_1| \geq C (R'')^{3/2} \delta(z), \text{  or  } |\pi(z)'-\pi(w)'| \geq C (R'')^{3/4} \delta(z)^{1/2},$$
where $C$ is an independent constant. In the first case, we obtain, using the triangle inequality:

\begin{align*}
|z_1-w_1| & \geq |\pi(z)_1-\pi(w)_1|-|\pi(z)_1-z_1|-|\pi(w)_1-w_1| \\
& \geq C (R'')^{3/2} \delta(z)- \delta(z)-\delta(w)\\
& = (C (R'')^{3/2}-1)\delta(z)-\delta(w).
\end{align*} 
Now, if $\delta(w)>\frac{(C (R'')^{3/2}-1)}{2}\delta(z),$ then we can proceed as in Case $2a$ (see below). Otherwise, we have 
$$|z_1-w_1|\geq \frac{(C (R'')^{3/2}-1)}{2}\delta(z),$$
and we are done in this case. 

In the second case, we obtain 

\begin{align*}
|z'-w'| & \geq |\pi(z)'-\pi(w)'|-|\pi(z)'-z'|-|\pi(w)'-w'| \\
& \geq C (R'')^{3/4} \delta(z)^{1/2}- \delta(z)-\delta(w)\\
& \geq (C (R'')^{3/4}-c)\delta(z)^{1/2}-\delta(w).
\end{align*} 
where $c$ is some constant that only depends on the domain $\Omega$ and the defining function $\rho.$ We can split into further subcases as before depending on the size of $\delta(w)$ to obtain the desired result.

\textbf{Case 1b:} 

The second sub-case of Case $1$ implies
$$\delta(w) \leq \frac{1}{R''} \delta(z),$$ so we are done in that case. 

Now we turn to Case $2$.

\textbf{Case 2a:}

First suppose that $h(z) \leq h(w).$ Then the condition reads
$$\sqrt{h(w)} \geq \frac{R''}{2} \sqrt{h(z)}$$
or equivalently
$$\delta(w) \geq \frac{(R'')^4}{16} \delta(z).$$
This is the bound we want. Indeed, 
Taylor expansion of $\rho$ shows that if $|z-w|^2+ |\langle z-w,\overline{\partial}\rho(z) \rangle|< R \delta(z)$, then we have

\begin{align*}
\delta(w) & \lesssim \rho(w)\\
& \leq |\rho(z)|+|\rho(z)-\rho(w)|\\
& \leq c(1+R) \delta(z),
\end{align*}
where $c$ is a constant that depends on the domain $\Omega.$ The contrapositive of this argument then shows that the bound we obtained provides a desired lower bound on $|z-w|^2+ |\langle z-w,\overline{\partial}\rho(z) \rangle|.$

\textbf{Case 2b:}

The final sub-case to consider is when $h(z) \geq h(w).$ In this case, we can directly verify that
$$\delta(w) \leq \frac{16}{(R'')^4} \delta(z).$$
\end{proof}

Now we are ready to prove (\ref{eq:tail}) in Lemma \ref{lemma:kernel}.
Let $\delta_0= \delta(z_0).$ We now apply Lemma \ref{kobayashishape}. We assume $n,p >1$ (for the case $n=1$, make the obvious modifications to the proof). Using either the previously mentioned Rudin-Forelli estimates or the main theorem in \cite{kerzman1972}, it is straightforward to see
	\[ \lim_{r \to 1^-} \sup_{\delta(z) \ge \ep_0} \|k_z^{p,\alpha}\|_{L^p_\alpha(\Omega \backslash Y(z,r)} =0 \]
so we can also assume that $\delta_0<\varepsilon_0.$ By a unitary rotation and translation, we can assume that the coordinates $w=(w_1,w')$ are centered at $z_0$ and  split into the complex normal and tangential directions at $\pi(z_0)$. 
 Moreover, we note that well-known asymptotics for the Bergman kernel in the strongly pseudoconvex case together with the Rudin-Forelli estimates (see \cite{peloso1994}, \cite{xiawang2021}) give the estimate in these coordinates, for any $w \in \Omega$:

\begin{align*}|k_{z_0}^{p,\alpha}(w)| & \lesssim \frac{\delta_0^{(n+1+\alpha)/p'}}{(\delta_0+|\rho(w)|+ |\langle z_0-w,\overline{\partial}\rho(z_0) \rangle|+|z_0-w|^2)^{(n+1+\alpha)}}    \\
& \lesssim \frac{\delta_0^{(n+1+\alpha)/p'}}{(\delta_0+|\rho(w)|+|w_1|+|w'|^2)^{(n+1+\alpha)}}.
\end{align*}

We additionally have
$$
|k_{z_0}^{p,\alpha}(w)|  \lesssim \frac{\delta_0^{(n+1+\alpha)/p'}}{(\delta_0+|\rho(w)|+ |\text{Im} \Psi(z_0,w)|+|z_0-w|^2)^{(n+1+\alpha)}} 
$$
We use the coordinates given in Lemma 2.6 in \cite{peloso1994}. In particular, there exist small positive numbers $\varepsilon_0'$ and $\delta_0'$ so that for each $z_0$ satisfying $|\rho(z_0)|< \varepsilon_0'$, there exists a $C^\infty$ diffeomorphism  $t(\cdot,z_0)$ defined on the Euclidean ball centered at $z_0$ with radius $\delta_0'$ so that the (real) Jacobian of $t(\cdot,z_0)$ is bounded above and the Jacobian determinant is bounded from below (with uniform bounds independent of $z_0$). Moreover, the coordinates $(t_1,t_2,t')=t(w,z_0)$ satisfy
$$t_1(w,z_0)=-\rho(w), \quad t_2(w,z_0)=\text{Im}\Psi(z_0,w).$$
Here $t_1 \in \mathbb{R}^{+}, t_2 \in \mathbb{R}$, and $t' \in \mathbb{C}^{n-1}.$
We may assume without loss of generality that $\varepsilon_0$ is sufficiently small so that if $z \in \Omega$ and $\delta(z)<  \varepsilon_0$, then we have $|\rho(z)|< \varepsilon_0'$, where $\varepsilon_0'$ is as in the Lemma. Choose $R_N \in (0,1)$ as in Lemma \ref{kobayashishape}. Taylor series arguments together with Lemma \ref{kobayashishape} and the fact that $t(w,z_0)$ is a diffeomorphism imply that there exists a $c>0$ so that if $d(z,w)\geq \tanh^{-1}R_N,$ then either $t_1+|t_2|+|t'|^2 \geq 3c N \delta_0$ or $t_1< \frac{c}{N} \delta_0.$ We then integrate in this coordinate system and split into the corresponding cases:
\begin{align*}
& \int_{\Omega \setminus Y(z_0,R_N)} |k_{z_0}^{p,\alpha}(w)|^p |\rho(w)|^{\alpha} \mathop{dA(w)}\\
& \lesssim \int_{{\Omega \setminus Y(z_0,R_N)} \cap \{|z_0-w| \geq \delta_0'\}} \delta_0^{(n+1+\alpha)(p-1)} |\rho(w)|^\alpha |K_{\alpha}(z_0,w)|^p \, dA(w)\\
& + \int_{c N \delta_0}^{\infty}  \int_{-\infty}^\infty \int_{0<|t'|<\infty}  \frac{ \delta_0^{(n+1+\alpha)(p-1)}t_1^\alpha}{(\delta_0+t_1+|t_2|+|t'|^2)^{p(n+1+\alpha)}} \, dA(t') \, dt_2 \,dt_1 \\
&  + \int_{0}^{\infty}  \int_{|t_2|> c N \delta_0} \int_{0<|t'|<\infty}  \frac{ \delta_0^{(n+1+\alpha)(p-1)}t_1^\alpha}{(\delta_0+t_1+|t_2|+|t'|^2)^{p(n+1+\alpha)}} \, dA(t') \, dt_2 \,dt_1\\
&  + \int_{0}^{\infty}  \int_{-\infty}^{\infty} \int_{\sqrt{c N \delta_0} <|t'|<\infty}  \frac{ \delta_0^{(n+1+\alpha)(p-1)}t_1^\alpha}{(\delta_0+t_1+|t_2|+|t'|^2)^{p(n+1+\alpha)}} \, dA(t') \, dt_2 \,dt_1\\
&  + \int_{0}^{\frac{c}{N} \delta_0}  \int_{-\infty}^\infty \int_{0<|t'|<\infty}  \frac{ \delta_0^{(n+1+\alpha)(p-1)}t_1^\alpha}{(\delta_0+t_1+|t_2|+|t'|^2)^{p(n+1+\alpha)}} \, dA(t') \, dt_2 \,dt_1.
\end{align*}

We will show  that the first term can be made as small as we want independently of $\delta_0.$ So let $\varepsilon>0$. Write $dA_\alpha=|\rho|^\alpha dA$ and note $dA_\alpha$ is a finite measure. The first term is easily seen to be controlled by $\delta_0^{(n+1+\alpha)(p-1)} C_{p, \alpha, \Omega} A_\alpha(\Omega \setminus Y(z_0, R_N)),$ where $C_{p,\alpha, \Omega}$ is a constant depending on only depending on $p, \alpha$ and $\Omega.$ If $\delta_0^{(n+1+\alpha)(p-1)} C_{p,\alpha,\Omega} A_\alpha(\Omega)<\varepsilon,$ then we are done. Otherwise $\delta_0 \geq \varepsilon'$, where $\varepsilon'$ depends only on $ \varepsilon, p, \alpha,$ and $\Omega.$ Since $\{z \in \Omega: \delta(z) \geq \varepsilon'\}$ is a compact subset of $\Omega$, we may choose $N$ (depending only on $\varepsilon'$) so that  $\delta_0^{(n+1+\alpha)(p-1)} C_{p,\alpha,\Omega} A_\alpha(\Omega \setminus Y(z_0, R_N))<\varepsilon.$

On the other hand, for the second term we have

\begin{align*}
& \int_{c N \delta_0}^{\infty}  \int_{0}^\infty \int_{0<|t'|<\infty}  \frac{ \delta_0^{(n+1+\alpha)(p-1)}t_1^\alpha}{(\delta_0+t_1+|t_2|+|t'|^2)^{p(n+1+\alpha)}} \, dA(t') \, dt_2 \,dt_1\\
& =  C_n \int_{c N \delta_0}^{\infty} \int_{0}^\infty \int_{0}^{\infty} \frac{ \delta_0^{(n+1+\alpha)(p-1)} t_1^\alpha r^{2n-3}}{(\delta_0+t_1+|t_2|+r^2)^{p(n+1+\alpha)}}  \,d r \,d t_2 \,d t_1\\
& =  \frac{C_n}{2} \int_{c N \delta_0}^{\infty} \int_{0}^\infty \int_{0}^{\infty} \frac{ \delta_0^{(n+1+\alpha)(p-1)} t_1^\alpha r'^{(n-2)}}{(\delta_0+t_1+|t_2|+r')^{p(n+1+\alpha)}}  \,d r' \,d t_2 \,d t_1\\
& =  \frac{C_n}{2} \int_{c N \delta_0}^{\infty} \int_{0}^\infty \int_{0}^{\infty}  \frac{ \delta_0^{(n+1+\alpha)(p-1)+\alpha+n-2} \left(t_1/\delta_0\right)^{\alpha}  \left(r'/\delta_0\right)^{(n-2)} \, d r' \, d t_2 \, d t_1}{\delta_0^{p(n+1+\alpha)}\left[1+\left(t_1/\delta_0\right)+\left(|t_2|/\delta_0\right)+\left(r'/\delta_0\right)\right]^{p(n+1+\alpha)}}  \\
& \leq  C_n' \int_{cN}^{\infty} \int_{0}^{\infty} \int_{0}^{\infty}  \frac{v_1 ^{\alpha} u^{n-2}}{\left[1+ v_1+v_2+u\right]^{p(n+1+\alpha)}}  \, d u \, d v_2 \, d v_1. \stepcounter{equation}\tag{\theequation}\label{eq1}
\end{align*}

It is a simple matter to check, using the fact that $p>1$ and $\alpha \geq 0,$ that 
$$\int_{0}^{\infty} \int_{0}^{\infty} \int_{0}^{\infty}  \frac{v_1 ^{\alpha} u^{n-2}}{\left[1+ v_1+v_2+u\right]^{p(n+1+\alpha)}}  \, d u \, d v_2 \, d v_1<\infty.$$
Then, using the Dominated Convergence Theorem, we deduce that \eqref{eq1} goes to $0$ as $N \rightarrow \infty.$

Using an entirely similar integration procedure, the third term can be bounded above by the following integral
$$C_n' \int_{0}^{\infty} \int_{cN}^{\infty} \int_{0}^{\infty}  \frac{v_1 ^{\alpha} u^{n-2}}{\left[1+ v_1+v_2+u\right]^{p(n+1+\alpha)}}  \, d u \, d v_2 \, d v_1$$

while the fourth term can be controlled by
$$C_n' \int_{0}^{\infty} \int_{0}^{\infty} \int_{cN}^{\infty}  \frac{v_1 ^{\alpha} u^{n-2}}{\left[1+ v_1+v_2+u\right]^{p(n+1+\alpha)}}  \, d u \, d v_2 \, d v_1$$

and finally the fifth term can be bounded by

$$C_n' \int_{0}^{\frac{c}{N}} \int_{0}^{\infty} \int_{0}^{\infty}  \frac{v_1 ^{\alpha} u^{n-2}}{\left[1+ v_1+v_2+u\right]^{p(n+1+\alpha)}}  \, d u \, d v_2 \, d v_1.$$
All of these integrals can be seen to approach $0$ as $N \rightarrow \infty$ independently of $z_0$ by the Dominated Convergence Theorem. This establishes the conclusion of Lemma \ref{lemma:kernel}.

\section{Lower Dimensional Sets}\label{sec:lower-dim}

In this section we will prove Theorems \ref{thm:lower-dim-kob} and \ref{thm:lower-dim-euclid}. The good/bad decomposition used to go from the local estimate to global estimate in Section \ref{subsec:suff} carries over to the lower-dimensional setting. So it is enough to prove the analogous local estimate. To do so, we will use the following estimates on the $\H^{2n-2}$ measure of the zero set of a holomorphic function.

\begin{lemma}\label{lemma:zero}
Let $Q \subset \C^n$ be a Euclidean cube  and let $\ell(Q)$ and $c(Q)$ denote its side length and center, respectively . For any $f \in \Hol(2Q)$,
	\[ \H^{2n-2}(\{f=0\} \cap Q) \le C \ell(Q)^{2n-2} \log \frac{\sup_{2Q}|f|}{\sup_{Q}|f|}. \]
Moreover, if $f(c(Q))=0$, then
	\[ \H^{2n-2}(\{f=0\} \cap Q) \ge c \ell(Q)^{2n-2}.\]
If $\Omega$ is a smoothly bounded strongly pseudoconvex domain, then there exists $C,\ep>0$ such that for all $z \in \Omega_{\ep}$ and $f \in \Hol(Y(z,R))$,
	\[ \H^{2n-2}(\{f=0\} \cap Y(z,r)) \le C \delta(z)^{n-1} \log \frac{\sup_{Y(z,R)}|f|}{\sup_{Y(z,r)}|f|} .\]
\end{lemma}
\begin{proof}
The first two statements are classical and we refer to the textbook \cite[p. 230]{ronkin-book}.
By the argument of (\ref{eq:sep}), for each $w \in Y(z,r)$, we obtain a suitably rotated and translated polydisc 
	\[ P^*(w) \subset Y(w,\tfrac{R-r}{2}) \subset Y(z,\tfrac{R+r}{2}) .\]  Let $\frac 14 P^*(z_0)$ denote the polydisc obtained from $P^*(z_0)$ by scaling its dimensions in each complex direction by a factor of $\frac{1}{4}.$
Using the fact that $Y(z,r)$ is contained in a large polydisc, the pigeonhole principle provides a $z_0 \in Y(z,r)$ and $c_0>0$ such that 
	\[ \H^{2n-2}(\{f=0\} \cap \frac 14 P^*(z_0)) \ge c_0 \H^{2n-2}(\{f=0\} \cap Y(z,r)).\]
Let $z_1 \in \frac 14 P^*(z_0)$ such that $|f(z_1)| = \sup_{\frac 14 P^*(z_0)} |f|$. Then, there exists $C>0$ from Jensen's formula for polydiscs \cite[Thm 4.2.5]{ronkin-book} such that
	\[ \H^{2n-2}(\{f=0\} \cap \frac 12 P^*(z_1)) \le C \delta(z)^{n-1}\log \frac{\sup_{ P^*(z_1)}|f|}{|f(z_1)|}.\]
On one hand, $\H^{2n-2}(\{f=0\} \cap \frac 12 P^*(z_1)) \ge \H^{2n-2}(\{f=0\} \cap \frac 14  P^*(z_0)) \ge c_0 \H^{2n-2}(\{f=0\} \cap Y(z,r))$. At the same time, we want to connect $|f(z_1)|$ with $\sup_{Y(z,r)}|f|$. Applying Lemma \ref{lemma:local} with $p=\infty$, 
	\[ |f(z_1)| = \sup_{\tfrac 14 P^*(z_0)} |f| \ge C^{-CN(\frac{R+r}{2},R)} \sup_{Y(z,\frac{R+r}{2})} |f| \ge C^{-CN(\frac{R+r}{2},R)} \sup_{Y(z,r)} |f|. \]
Therefore,
	\[\begin{aligned} \log \frac{\sup_{P^*(z_1)}|f|}{|f(z_1)|} &\le \log\frac{\sup_{Y(z,R)}|f|}{C^{-CN(\frac{R+r}{2},R)} \sup_{Y(z,r)}|f|} \\
		&= \log \frac{\sup_{Y(z,R)}|f|}{\sup_{Y(z,r)}|f|} + C \log \frac{\sup_{Y(z,R)}|f|}{\sup_{Y(z,\frac{R+r}{2})}|f|}\log C\\
		& \le (1+C \log C) \log \frac{\sup_{Y(z,R)}|f|}{\sup_{Y(z,r)}|f|}. \end{aligned}\]
\end{proof}

We have the following local estimates for Kobayashi balls.

\begin{lemma}\label{lemma:local-lower}
Let $\Omega \subset \C^n$ be a smoothly bounded strongly pseudoconvex domain, $R>r>0$, $1 \le p \le \infty$, $N^*,\nu>0$. 
There exists $C,c,\ep>0$ such that for any $0<\gamma<c$,
	\begin{equation}\label{eq:remez-lower} \|f\|_{L^p(Y(z,r))} \le \left( \frac{C}{\gamma}  \right)^{C} \|f\|_{L^p(E,\H^{2n-2+\nu})}\end{equation}
for all $z \in \Omega_\ep$, $f \in \Hol(Y(z,R))$ with $N_p(z,f,r,R) \le N^*$, and $E \subset Y(z,r)$ satisfying
	\[ \dfrac{|Y(z,r)|^{\frac{n-1}{n(n+1)}}\H^{2n-2+\nu}(E)}{|Y(z,r)|^{(2n-2+\nu)/2n}} \ge \gamma. \]
\end{lemma}

One might be bothered by the extra factor $|Y_r|^{\frac{n-1}{n(n+1)}}$, however it occurs because the Hausdorff measure does not scale according to the determinant of an affine map, but rather according to the map's most extreme directions. For example, if $D$ is a diagonal matrix with all entries $\lambda_i$, then $\min\{\lambda_i\} \H^1(E) \le \H^1(DE) \le \max\{\lambda_i\} \H(E)$ and one can find sets $E_{max}$ and $E_{min}$ which attain the upper and lower bounds, respectively.

\begin{proof}
We will only prove the case $p=\infty$. One may go back to any $1 \le p \le \infty$ by repeating the steps at the end of the proof of Lemma \ref{lemma:local}. Normalize $f$ so that $\sup_{Y(z,r)}|f|=1$. In this way, it is enough to obtain a lower bound on $\sup_{E}|f|$. 
Using the local estimate (Lemma \ref{lemma:local}), we will connect the Hausdorff measure of small sublevel sets 
	\[ F_a = Y(z,r) \cap \{ |f| \le e^{-a} \} \] 
to that of the zero set. Cover $Y(z,r)$ with $K$ disjoint cubes of side length $\ell \sim (|Y_r|/K)^{1/2n}$  contained in $Y(z,R)$ . If $K$ is chosen appropriately, and $a$ is large, then any cube that intersects $F_a$ contains a zero of $f$. Indeed, if $f$ does not have a zero in $Q$, but $Q$ intersects $F_a$, then by Harnack's inequality,
	\[ \sup_{Q} |f| \le c \inf_Q |f| \le ce^{-a}. \]
On the other hand, applying the local estimate Lemma \ref{lemma:local},
	\[ (CK)^{-C(N+1)} \le  \sup_{Q}|f|, \quad N = N_\infty(r,R)\]
we obtain a contradiction for $K = C^{-1}e^{ a/(2C(N+1))}$ and then $a$ large. 

Now, $\H^{2n-2+\nu}(F_a) \le \sum_{Q \cap F_a \ne \varnothing} \ell^{2n-2+\nu}$ so it remains to estimate the number of cubes intersecting $F_a$, which we have just shown to be bounded above by the number of cubes with zeroes.  Replace each cube $Q$ with its double $2Q$. If $f$ has a zero on $Q$ then we can find $Q^* \subset 2Q$ with $f(c(Q^*))=0$ and the side length of $Q^*$ is still $\ell$. In this way, we can apply the second statement in Lemma \ref{lemma:zero} to get $\mathcal H^{2n-2}(\{f=0\} \cap 2Q) \ge \mathcal H^{2n-2}(\{f=0\} \cap Q^*) \gtrsim \ell^{2n-2}$. Since $K$ is large, for $s = \frac{R+r}{2}$, $\cup_{Q \cap Y(z,r) \ne \varnothing} 2Q \subset Y(z,s)$, so 
	\[ \H^{2n-2}(\{f=0\} \cap  Y(z,s) ) \ge \frac{1}{2^{2n}} \sum_{Q \cap F_a \ne \varnothing} \H^{2n-2}(\{f=0\} \cap 2Q) \gtrsim  \sum_{Q \cap F_a \ne \varnothing} \ell^{2n-2}. \]

From Lemma \ref{lemma:zero}, we also have $\H^{2n-2}(\{f=0\} \cap Y(z,s)) \le C \delta(z)^{n-1} N(s,R)$ but $N(s,R) \le N(r,R)=:N$ since $r < s$. Altogether, this implies
	\[ \begin{aligned} \H^{2n-2+\nu}(F_a) &\le C\ell^{\nu} N \delta(z)^{n-1} \\
				&\le CN \delta(z)^{n-1} |Y(z,r)|^{\nu/2n} K^{-\nu/2n} \\
				&\le CN|Y(z,r)|^{q^*} K^{-\nu/2n}, \end{aligned}\]
where, recalling the fact that $|Y_r| \sim \delta(z)^{n+1}$ from Lemma \ref{lemma:kob-geo}, 
	\[ q^*= 1-\frac{2}{n+1}+\frac{\nu}{2n} 
	= \frac{2n-2+\nu}{2n} - \frac{n-1}{n(n+1)}. \]
Moreover, plugging in the value of $K$ from above, we have shown that for $a$ large enough,
	\begin{equation}\label{eq:level} \H^{2n-2+\nu}(\{|f| \le e^{-a}\} \cap Y(z,r)) \le C N |Y(z,r)|^{q^*}e^{-\nu a/ 4 C n(N+1) }.\end{equation}
Now we can prove the lemma. Let $E \subset Y(z,r)$ with $\H^{2n-2+\nu}(E) \ge \gamma |Y(z,r)|^{q^*}$. Pick $a$ so that the $CNe^{-\nu a/4 C n(N+1)}=\gamma/2< \gamma $ ($\gamma$ must be small in order that $a$ be large enough to apply (\ref{eq:level}). This determines $c$). Since $a$ is large, (\ref{eq:level}) forces $E\backslash F_a$ to be non-empty and thus
	 \[ \sup_{E}|f| \ge e^{-a} = \left( \frac{\gamma}{2CN} \right)^{4Cn(N+1)\nu^{-1}}\] 
which proves (\ref{eq:remez-lower}) when $p=\infty$.
\end{proof}

\begin{cor}\label{cor:2n-2}
Given $N^*>0$, there exists $C,\tilde \gamma,\ep >0$ such that (\ref{eq:remez-lower}) holds for $\nu =0$ in the form
	\[ \|f\|_{L^p(Y(z,r))} \le C \|f\|_{L^p(E,\H^{2n-2})}\]
for all $z \in \Omega_\ep$, $f \in \Hol(Y(z,R))$ with $N_p(z,f,r,R) \le N^*$, and $E \subset Y(z,r)$ satisfying
	\[ \dfrac{|Y(z,r)|^{\frac{n-1}{n(n+1)}}\H^{2n-2}(E)}{|Y(z,r)|^{(2n-2)/2n}} \ge \tilde \gamma. \]

\end{cor}

\begin{proof}
Looking at the previous proof, we did not use the fact $\nu>0$ until after (\ref{eq:level}). Picking up there, with $\nu =0$, if
	\[ \tilde \gamma \ge 2 CN^*, \]
then one still obtains that $E \backslash F_a$ is non-empty and so $\sup_E |f| \ge e^{-a}$.
\end{proof}

The same strategy proves Theorem \ref{thm:lower-dim-euclid} by replacing the Kobayashi balls with Euclidean balls or cubes. One ought to perform a Whitney decomposition of $\Omega$. The analogues of Lemmas \ref{lemma:kob-geo} and \ref{lemma:kob-mean} are automatic for Euclidean cubes or balls. Also, the analogous local estimate, up to the form of the constant, is proved by Logunov and Malinnikova in a much more general setting \cite{logunov2018quantitative}.

\section{Reverse Carleson Measures}\label{sec:reverse}
Recall that a positive, finite Borel measure $\mu$ is said to be Carleson on $A^p(\Omega)$ if there exists a constant $C$ so that for all $f \in A^p(\Omega)$, 
$$\int_{\Omega}|f|^p \, d\mu \leq C \int_{\Omega} |f|^p \, dA.$$ 
Carleson measures on smoothly bounded, strongly pseuedoconvex domains were characterized by Abate and Saracco in \cite{abatesaracco2011}. It turns out the following condition is both necessary and sufficient for $\mu$ to be Carleson on $A^p(\Omega)$ for all $p>0$: there exists an $r \in (0,1)$ so that

$$\sup_{z \in \Omega} \frac{\mu(Y(z,r))}{|Y(z,r)|}<\infty.$$

Moreover, if the above condition holds for some $r$, it holds for all $r \in (0,1).$ Therefore, a positive finite Borel measure $\mu$ is Carleson on $A^p(\Omega)$ if and only if the following ``norm'' is finite:

$$\|\mu\|_{\mathcal{C}}:= \sup_{z \in \Omega} \frac{\mu(Y(z,1/2))}{|Y(z,1/2)|}.$$
The choice of $1/2$ is not important by the remark above.  

Conversely, recall reverse Carleson measures were defined in \ref{subsec:reverse}. In analogy with Luecking in \cite{luecking1985forward}, we are able to provide a sufficient condition for a measure $\mu$ to be reverse Carleson on $A^p_\alpha(\Omega)$ in terms of the relative density condition for $1 \leq p < \infty.$ The sufficient condition includes the assumption that $\|\mu_\alpha\|_{\mathcal{C}}<\infty,$ where $d\mu_\alpha= |\rho|^{-\alpha} d\mu$ (so $\mu_\alpha$ is absolutely continuous with respect to $\mu$).  Given a measure $\mu$ and $s \in (0,1)$, we introduce the notation 	
	\[ \bar \mu_s(z):= \frac{\mu(Y(z,s))}{|Y(z,s)|}.\]

\begin{theorem} \label{reversecarleson}
Let $r_0 >0,$ $\alpha>-1$, and $1 \le p < \infty.$ There exists constants $K,q>1$, depending only on these parameters, so that for any $0<\gamma,\varepsilon<1$ and $s \leq K^{-1} \varepsilon^{1/p} \gamma^q$ the following property holds: There exists $C>0$ such that for all positive finite Borel measures $\mu$ satisfying 
	\begin{itemize}
	\item $\|\mu_\alpha\|_{\mathcal{C}}<\infty$,
	\item The set $G=\{z \in \Omega:  (\overline {\mu_\alpha})_{s} (z)>\varepsilon \|\mu_\alpha\|_{\mathcal{C}}\}$ is relatively dense with respect to the parameters $r_0$ and $\gamma$, 
	\end{itemize}
$$ \int_{\Omega} |f|^p |\rho|^\alpha \, dA \leq C  \int_{\Omega} |f|^p \, d\mu \quad \forall f \in A^p_\alpha(\Omega).$$  

\end{theorem}

To prove this theorem, we will first refine some of the properties concerning the Kobayashi metric from Section \ref{sec:pseudo}. The first we recall is from \cite[Lemma 6]{li1992}: If $\tanh^{-1}(r)\leq 1,$ then there exists constants $a(r), b(r), A, B$,  where $A, B \geq 1$  and only depend on $\Omega$, so 
\begin{equation} P(0, a(r) \delta(z), b(r) \delta(z)^{1/2}) \subset U(z)(Y(z,r)-z) \subset P(0,A \tanh^{-1}(r) \delta(z), B\tanh^{-1}(r) \delta(z)^{1/2}).\label{polydisccontainment} \end{equation} 

We now track the dependence on $r$ in Lemma \ref{lemma:kob-geo}:

\begin{proposition} \label{comparablesize}
There exist constants $c_1$, $C_1$ so that for $r$ with $\tanh^{-1}(r)<1$, $z \in \Omega$, and $w \in Y(z,r),$

$$ c_1 |\rho(w)| \leq |\rho(z)| \leq C_1 |\rho(w)|$$ 
and
$$c_1 |Y(z,r)| \leq |Y(w,r)| \leq C_1 |Y(z,r)|.$$
Moreover, there exist constants $c_2$, $C_2$ so that if $\tanh^{-1}(r)<1$ and $z \in \Omega,$

$$c_2 r^{2n} \delta(z)^{n+1} \leq |Y(z,r)| \leq C_2 r^{2n} \delta(z)^{n+1}.$$
\begin{proof}
The first display is \cite[Lemma 1.2]{abatesaracco2011}.  For the other displays, the upper polydisc containment gives the upper bound and \cite[Lemma 1.1]{abatesaracco2011} gives the lower bound. 

\end{proof}
\end{proposition}

\begin{proposition}\label{differencequotient}
Let $0<R < 1$,  $A,B$ be as in \eqref{polydisccontainment} , $\tanh^{-1}(s) \leq \frac{1}{3AB} \min\{\frac{1}{2},a(\frac{R}{2}),b(\frac{R}{2})\},$ and $1 \le p \le \infty$. Then there exists $C_R>0,$ independent of $s$, such that for all holomorphic $f$ and $z \in \Omega$,

	$$ \sup_{w \in Y(z,s)} \frac{|f(z)-f(w)|}{d(z,w)} \le C_R \langle f \rangle_{Y(z,R),p}. $$
	
\begin{proof}
We will apply Cauchy's integral formula for polydiscs to the function $f$. Let $d_k(z)=\delta(z)$ if $k=1$ and $\delta(z)^{1/2}$ otherwise.  By composing $f$ with an appropriate unitary transformation and translation, which preserves the Kobayashi metric, we can assume without loss of generality that the standard coordinates $z_1,\dots,z_n$ coincide with the suitably rotated coordinates centered at $z$ (i.e. we assume $z=0$).  Write the polydisc $P(0,\delta(0) r_1,\delta(0)^{1/2}r_2)= D(0,d_1(0) r_1) \times D(0,d_2(0) r_2) \times \dots \times D(0, d_n(0)r_2).$ Let $r=\frac{1}{3}\min\{\frac{1}{2},a(\frac{R}{2}),b(\frac{R}{2})\}$.  Note that we have, by our choice of $s$,
$$Y(0,r) \subset P(0,  r \delta(0),  r \delta(0)^{1/2}) \subset  P(0, 2 r \delta(0), 2 r \delta(0)^{1/2}) \subset \Omega.$$

 Applying the formula to the polydisc $P(0, 2r \delta(0), 2r \delta(0)^{1/2})$, we have 

\begin{align*}
 f(0)-f(w) = \frac{1}{(2\pi i)^n}\int_{\substack{\partial D(0,2d_1(0) r) \times \dots \\ \times \partial D(0,2 d_n(0) r)}  }  f(\zeta)\left( \frac{1}{\prod_{k=1}^n\zeta_k}-\frac{1}{\prod_{k=1}^n(\zeta_k-w_k)}\right) d\zeta.  
\end{align*}

It is clear we must estimate

$$ \left| \frac{1}{\prod_{k=1}^n \zeta_k}-\frac{1}{\prod_{k=1}^n(\zeta_k-w_k)}\right|.$$

Note that by our choice of polydisc, we have the inequalities 

 $$r d_k(0) \leq |\zeta_k-w_k| \leq 2 r d_k(0)$$
for $\zeta_k \in  \partial D(0,2 d_k(0)r)$ and $1 \leq k \leq n$. And by definition, $|\zeta_k|= 2 r d_k(0)$ for $1 \leq k \leq n.$  

Therefore, we estimate

\begin{align*}
& \left| \frac{1}{\prod_{k=1}^n \zeta_k}-\frac{1}{\prod_{k=1}^n(\zeta_k-w_k)}\right|\\
&= \left| \frac{\prod_{k=1}^n(\zeta_k-w_k)-\prod_{k=1}^n(\zeta_k) }{\prod_{k=1}^n\zeta_k(\zeta_k-w_k)}\right|\\
& = \left| \frac{\prod_{k=1}^n(\zeta_k-w_k)-\zeta_1\prod_{k=2}^n(\zeta_k-w_k)+\zeta_1\prod_{k=2}^n(\zeta_k-w_k)-\prod_{k=1}^n\zeta_k }{\prod_{k=1}^n\zeta_k(\zeta_k-w_k)}\right|\\
& \leq  C_R \left(\frac{|w_1|}{d_1(0)\delta(0)^{(n+1)/2}}\right)+ \left|\frac{\zeta_1\prod_{k=2}^n(\zeta_k-w_k)-\prod_{k=1}^n\zeta_k }{\prod_{k=1}^n\zeta_k(\zeta_k-w_k)}\right|.
\end{align*}

Iterating this process in an obvious way, we obtain 

$$\left| \frac{1}{\prod_{k=1}^n \zeta_k}-\frac{1}{\prod_{k=1}^n(\zeta_k-w_k)}\right| \leq \frac{C_R}{\delta(0)^{(n+1)/2}}\sum_{k=1}^{n} \frac{|w_k|}{ d_k(0)}.$$

We now claim 
$$\sum_{k=1}^{n} \frac{|w_k|}{ d_k(0)} \leq C d(0,w). $$

To see this, let $r_0=d(0,w)<\tanh^{-1}(s).$ Then $w \in Y(0, \tanh(2 r_0)),$ which by the polydisc containment property implies $w \in P(0,2A r_0 \delta(0),  2B r_0 \delta(0)^{1/2})$. This then easily implies
$$\sum_{k=1}^{n} \frac{|w_k|}{ d_k(0)} \leq C r_0= C d(0,w).$$

Straightforward estimation then implies

\[ \begin{aligned} |f(0)-f(w)| &\leq C_R d(0,w) \sup_{\zeta \in P(0, 3r \delta(0), 3r \delta(0)^{1/2})} |f(\zeta)| \\
	&\leq  C d(0,w) \sup_{\zeta \in P(0, a(\frac{R}{2}) \delta(0) ,b(\frac{R}{2})\delta(0)^{1/2})} |f(\zeta)|. \end{aligned} \]

Note that $ P(0, a(\frac{R}{2})\delta(0) ,b(\frac{R}{2}) \delta(0)^{1/2}) \subset Y(0,\frac{R}{2})$, so the result follows by applying the mean value property on $Y(0,\frac{R}{2})$ (see Lemma \ref{lemma:kob-mean}). 

\end{proof}
\end{proposition}

Let $\chi_{s}(w,z)= \begin{cases} 1 & d(w,z)<s \\ 0 & d(w,z)\geq s \end{cases}$. We now can prove the following lemma:

\begin{lemma}\label{reversecarlesonlemma}
Let $1 \le p< \infty$, $\alpha>-1$, $s>0$ with $\tanh^{-1}(s) \leq \frac{1}{3AB} \min\{\frac{1}{2},a(1/4),b(1/4)\}$, and $\mu$ be a measure satisfying $\|\mu\|_{\mathcal{C}}<\infty.$  Then there exists a constant $C$, independent of $s$, so that there holds for all $f \in A^p_\alpha(\Omega)$:

\begin{equation}\int_{\Omega} \int_{\Omega} \frac{\chi_s(w,z)}{|Y(z,s)|} |f(w)-f(z)|^p |\rho(w)|^\alpha \,dA(w) \,d\mu(z) \leq C s^p \|\mu\|_{\mathcal{C}} \int_{\Omega} |f|^p |\rho|^\alpha \,dA. \end{equation} \label{lemmaequation}

\begin{proof} We estimate, using Propositions \ref{comparablesize} and \ref{differencequotient} with $R=1/2$:

\begin{align*}
\int_{\Omega} \frac{\chi_s(w,z)}{|Y(z,s)|}|f(w)&-f(z)|^p |\rho(w)|^\alpha \,dA(w)  \\
&\leq \sup_{w \in Y(z,s)}d(z,w)^p \int_{\Omega} \frac{\chi_s(w,z)}{|Y(z,s)|}\left(\frac{|f(w)-f(z|)}{d(z,w)}\right)^p |\rho(w)|^\alpha \,dA(w) \\
& \leq  Cs^{p} \int_{\Omega} \frac{\chi_s(w,z)}{|Y(z,s)|}\frac{1}{|Y(z,1/2)|}\int_{Y(z,1/2)} |f(\zeta)|^p |\rho(\zeta)|^\alpha \,dA(\zeta)\,dA(w)\\
& =  C s^{p} \frac{1}{|Y(z,1/2)|}\int_{Y(z,1/2)} |f(\zeta)|^p |\rho(\zeta)|^\alpha \,dA(\zeta).
\end{align*}

Then integrate with respect to the measure $\mu$ in the variable $z$ on both sides and apply Fubini's theorem:

\begin{align*}
& \int_{\Omega} \int_{\Omega} \frac{\chi_s(w,z)}{|Y(z,s)|} |f(w)-f(z)|^p |\rho(w)|^\alpha \,dA(w) \,d\mu(z)\\
 & \leq C s^{p} \int_{\Omega} |f(\zeta)|^p |\rho(\zeta)|^\alpha \left( \int_{Y(\zeta,1/2)} \frac{1}{|Y(z,1/2)|} \, d\mu(z)\right) \,dA(\zeta).
\end{align*}

Note that for $z \in Y(\zeta, 1/2)$, Proposition \ref{comparablesize} implies that 
$$
\frac{1}{|Y(z,1/2)|} \leq \frac{1}{c_1|Y(\zeta,1/2)|}.
$$

Thus, the last integral is dominated by

$$ \frac{C}{c_1} s^{p} \int_{\Omega} |f(\zeta)|^p |\rho(\zeta)|^\alpha \left( \int_{Y(\zeta,1/2)} \frac{1}{|Y(\zeta,1/2)|} \, d\mu(z)\right) \,dA(\zeta) \leq C s^{p} \|\mu\|_{\mathcal{C}} \int_{\Omega} |f|^p |\rho|^\alpha \,dA.$$

\end{proof} 
\end{lemma} 

\begin{proof}[Proof of Theorem \ref{reversecarleson}]
We follow the general strategy of Luecking in \cite{luecking1985forward}. Choose $s$ as in Lemma \ref{reversecarlesonlemma} (we will see another constraint on $s$ at the end of the proof). Raise both sides of \eqref{lemmaequation} with measure $\mu_\alpha$ to the $1/p$ power and use the reverse triangle inequality to obtain
\begin{align*}
& \left(\int_{\Omega}\int_{\Omega} \frac{\chi_s(w,z)}{|Y(z,s)|} |f(w)|^p |\rho(w)|^\alpha \,dA(w) \,d\mu_\alpha(z)\right)^{1/p}- \left(\int_{\Omega}\int_{\Omega} \frac{\chi_s(w,z)}{|Y(z,s)|} |f(z)|^p |\rho(w)|^\alpha \,dA(w) \,d\mu_\alpha(z)\right)^{1/p}\\
& \leq \left( C s^{p} \|\mu_\alpha\|_{\mathcal{C}} \int_{\Omega} |f|^p |\rho|^\alpha \,dA\right)^{1/p}.
\end{align*}

Notice that the subtracted term is dominated by, after integration in $w$ and applying Proposition \ref{comparablesize}, $C_1^\alpha \left(\int_{\Omega} |f|^p \, d\mu \right)^{1/p}.$ On the other hand, for the first term, apply Fubini's theorem and apply the following estimate using Lemma \ref{comparablesize}:
\begin{align*}
\int_{\Omega} \frac{\chi_s(w,z)}{|Y(z,s)|} \,d\mu_\alpha(z) & = \int_{Y(w,s)} \frac{1}{|Y(z,s)|} \, d\mu_\alpha(z)\\
& \geq \frac{1}{C_1} \int_{Y(w,s)} \frac{1}{|Y(w,s)|}\, d\mu_\alpha(z)\\
& = \frac{1}{C_1}  \frac{\mu_\alpha(Y(w,s))}{|Y(w,s)|}\\
& \geq \frac{1}{C_1}  \varepsilon \|\mu_\alpha\|_{\mathcal{C}} \chi_{G}(w).
\end{align*} 

Inserting this estimate, we get that the first integral is bounded below by 

\begin{align*}
\frac{1}{C_1^{1/p}}  \|\mu_\alpha\|_{\mathcal{C}}^{1/p} \varepsilon^{1/p} \left(\int_{G} |f|^p |\rho|^\alpha \,dA \right)^{1/p} \geq C' \gamma^{q}  \|\mu_\alpha\|_{\mathcal{C}}^{1/p} \varepsilon^{1/p}\left(\int_{\Omega} |f|^p |\rho|^\alpha \,dA \right)^{1/p},
\end{align*}
applying Theorem \ref{thm:main} and the subsequent comments. Therefore, we get the inequality:
\begin{align*}
& C' \gamma^{q} \|\mu_\alpha\|_{\mathcal{C}}^{1/p} \varepsilon^{1/p}\left(\int_{\Omega} |f|^p |\rho|^\alpha \,dA \right)^{1/p}- C_1^\alpha\left(\int_{\Omega} |f|^p \, d\mu\right)^{1/p}\\
& \leq  C^{1/p} s \|\mu_\alpha\|_{\mathcal{C}}^{1/p}\left( \int_{\Omega} |f|^p |\rho|^\alpha \,dA\right)^{1/p}
\end{align*}

Now, as long as $s$ is chosen so $C^{1/p} s\leq  \frac{C'}{2} \gamma^{q} \varepsilon^{1/p}$, we obtain 

$$\left(\int_{\Omega} |f|^p |\rho|^\alpha \,dA \right)^{1/p} \leq \frac{C_1^\alpha}{\|\mu_\alpha\|_{\mathcal{C}}^{1/p}(C'\gamma^{q} \varepsilon^{1/p}-C^{1/p}s)}\left(\int_{\Omega} |f|^p \, d\mu\right)^{1/p},$$

which establishes the result.

\end{proof}

We now apply Theorem \ref{reversecarleson} to prove the point sampling result, Theorem \ref{thm:sampling}.

\begin{proof}[Proof of Theorem \ref{thm:sampling}]
Again, we follow ideas from \cite{luecking1985forward}. Appealing to Lemma \ref{comparablesize}, we can define the finite absolute constants

$$C_3:= \frac{C_2}{c_2} \left(\sup_{0<R<\frac{1}{4}} \frac{R}{\tanh\left(\frac{\tanh^{-1}(R)}{5}\right)}\right)^{2n},$$

$$C_4:= \frac{C_2}{c_2}  \left[2\tanh(\tanh^{-1}(1/4)+\tanh^{-1}(1/2))\right]^{2n}.$$

Choose $R<1/4$ as in the proof of Theorem \ref{reversecarleson} with respect to parameters $\varepsilon=\frac{c_1}{C_3 C_4}, \gamma,$ and $r_0.$ Let $r= \tanh(\frac{\tanh^{-1}(R)}{5})$ and suppose that $G=\bigcup_{j=1}^{\infty} Y(a_j,r)$ is relatively dense with respect to the parameters $\gamma, r_0.$  Using a standard covering lemma, we can select a subsequence $b_k=a_{j_k}$ so that $Y(b_k,r)$ are pairwise disjoint and $\bigcup_{j=1}^{\infty} Y(a_j,r) \subset \bigcup_{k=1}^{\infty} Y(b_k,R).$ It is immediate that $G' = \bigcup_{k=1}^{\infty} Y(b_k,R)$ is relatively dense. 

Let $\mu:=\sum_{k=1}^{\infty} |Y(b_k,R)| |\rho(b_k)|^\alpha \delta_{b_k}.$ We now claim that $\mu_\alpha= \sum_{k=1}^{\infty} |Y(b_k,R)| \delta_{b_k}$ is Carleson. To see this, note that if  $z \in \Omega$ and $b_k \in Y(z,1/2)$, then $Y(b_k,r) \subset Y(z,\tanh(\tanh^{-1}r+ \tanh^{-1}(1/2))).$ Then we have, using the pairwise disjointness of the $Y(b_k,r)$ and Lemma \ref{comparablesize}:

\begin{align*}
\frac{\mu_\alpha(Y(z,1/2))}{|Y(z,1/2)|} & = \sum_{b_k \in Y(z,1/2)}\frac{|Y(b_k,R)|}{|Y(z,1/2)|} \\
& \leq C_3\sum_{b_k \in Y(z,1/2)}\frac{|Y(b_k,r)|}{|Y(z,1/2)|}\\
& \leq C_3\frac{|Y(z,\tanh(\tanh^{-1}r+ \tanh^{-1}(1/2)))|}{|Y(z,1/2)|}\\
& \leq C_3 C_4.
\end{align*}
In particular $\|\mu_\alpha\|_{\mathcal{C}} \leq C_3 C_4.$
Finally, we claim that 
$$G'= \{z: \bar {\mu_\alpha}_R(z) \geq \varepsilon \|\mu_\alpha\|_{\mathcal{C}}\},$$
which according to Theorem \ref{reversecarleson} implies that $\mu$ is reverse Carleson and gives the result.

To establish the claim, note that

\begin{align*}
\frac{\mu_\alpha(Y(z,R))}{|Y(z,R)|} & =  \sum_{b_k \in Y(z,R)} \frac{|Y(b_k,R)|}{|Y(z,R)|}.
\end{align*}

If $z \in G'$, then $z \in Y(b_{k_0},R)$ for at least one $k_0$ and hence the average of $\mu_\alpha$ on $Y(z,R)$ is bounded below by $c_1$ by Proposition \ref{comparablesize}. 

Conversely, if $z \notin G'$, then $b_k \notin Y(z,R)$ for any $k$ and hence the average of $\mu_\alpha$ on $Y(z,R)$ is 0. This proves the claim and the theorem. 
\end{proof}

\bibliographystyle{abbrv}
{\small\bibliography{references.bib}}

\begin{thebibliography}{10}

\bibitem{abate1989iteration}
M.~Abate.
\newblock {\em Iteration theory of holomorphic maps on taut manifolds}.
\newblock Mediterranean Press, 1989.

\bibitem{abatesaracco2011}
M.~Abate and A.~Saracco.
\newblock Carleson measures and uniformly discrete sequences in strongly
  pseudoconvex domains.
\newblock {\em J. Lond. Math. Soc. (2)}, 83(3):587--605, 2011.

\bibitem{baloghbonk2000}
Z.~M. Balogh and M.~Bonk.
\newblock Gromov hyperbolicity and the {K}obayashi metric on strictly
  pseudoconvex domains.
\newblock {\em Comment. Math. Helv.}, 75(3):504--533, 2000.

\bibitem{brudnyi1999local}
A.~Brudnyi.
\newblock Local inequalities for plurisubharmonic functions.
\newblock {\em Ann. Math.}, 149(2):511--533, 1999.

\bibitem{calzi2021carleson}
M.~Calzi and M.~M. Peloso.
\newblock Carleson and reverse {C}arleson measures on homogeneous {S}iegel
  domains.
\newblock {\em Complex Anal. Oper. Theory}, 16(1), 2022.

\bibitem{carleman1939probleme}
T.~Carleman.
\newblock Sur un probl{\`e}me d'unicit{\'e} pour les syst{\`e}mes
  d'{\'e}quations aux d{\'e}riv{\'e}es partielles {\`a} deux variables
  ind{\'e}pendantes.
\newblock {\em Ark. Mat. Astr. Fys.}, 26, 1939.

\bibitem{cuc2018}
{\v{Z}}.~{\v{C}}u{\v{c}}kovi{\'{c}}, S.~{\c{S}}ahuto{\u{g}}lu, and Y.~E.
  Zeytuncu.
\newblock A local weighted {A}xler-{Z}heng theorem in {$\mathbb C^n$}.
\newblock {\em Pacific J. Math.}, 294(1):89--106, 2018.

\bibitem{david2021carleson}
G.~David, L.~Li, and S.~Mayboroda.
\newblock Carleson measure estimates for the {G}reen function.
\newblock {\em Arch. Ration. Mech. Anal.}, 243(3):1525--1563, 2022.

\bibitem{fricain2015survey}
E.~Fricain, A.~Hartmann, and W.~T. Ross.
\newblock A survey on reverse {C}arleson measures.
\newblock In {\em Harmonic Analysis, operator theory, function theory, and
  applications}, volume~19, pages 91--123. Theta, Bucharest, 2017, 2015.

\bibitem{gilbarg2001elliptic}
D.~Gilbarg and N.~S. Trudinger.
\newblock {\em Elliptic Partial Differential Equations of Second Order}, volume
  224.
\newblock Springer Science \& Business Media, 2001.

\bibitem{hartmann2021dominating}
A.~Hartmann, D.~Kamissoko, S.~Konate, and M.-A. Orsoni.
\newblock Dominating sets in {B}ergman spaces and sampling constants.
\newblock {\em J. Math. Anal. Appl.}, 495(2):124755, 2021.

\bibitem{hormander1973introduction}
L.~H\"ormander.
\newblock {\em An introduction to complex analysis in several variables}.
\newblock Elsevier, 1973.

\bibitem{hormander-books}
L.~H\"ormander.
\newblock {\em The Analysis of linear partial differential operators I-IV},
  volume 257 of {\em Grundlehren}.
\newblock Springer-Verlag, 1983-1985.

\bibitem{zhu2016}
Z.~Hu, X.~Lv, and K.~Zhu.
\newblock Carleson measures and balayage for {B}ergman spaces of strongly
  pseudoconvex domains.
\newblock {\em Math. Nachr.}, 289(10):1237--1254, 2016.

\bibitem{kerzman1972}
N.~Kerzman.
\newblock The {B}ergman kernel function. {D}ifferentiability at the boundary.
\newblock {\em Math. Ann.}, 195:149--158, 1972.

\bibitem{kovrijkine01}
O.~Kovrijkine.
\newblock Some results related to the {L}ogvinenko-{S}ereda theorem.
\newblock {\em Proc. Amer. Math. Soc.}, 129(10):3037--3047, 2001.

\bibitem{krantz2001function}
S.~G. Krantz.
\newblock {\em Function theory of several complex variables}, volume 340.
\newblock American Mathematical Soc., 2001.

\bibitem{krantz1988bloch}
S.~G. Krantz and D.~Ma.
\newblock Bloch functions on strongly pseudoconvex domains.
\newblock {\em Indiana Univ. Math. J.}, 37(1):145--163, 1988.

\bibitem{lebeau19}
G.~Lebeau and I.~Moyano.
\newblock Spectral inequalities for the {S}chr{\"o}dinger operator.
\newblock {\em arXiv preprint arXiv:1901.03513}, 2019.

\bibitem{li1992}
H.~Li.
\newblock B{MO}, {VMO} and {H}ankel operators on the {B}ergman space of
  strongly pseudoconvex domains.
\newblock {\em J. Funct. Anal.}, 106(2):375--408, 1992.

\bibitem{logunov2018quantitative}
A.~Logunov and E.~Malinnikova.
\newblock Quantitative propagation of smallness for solutions of elliptic
  equations.
\newblock In {\em Proceedings of the International Congress of
  Mathematicians—Rio de}, volume~2, pages 2357--2378. World Scientific, 2018.

\bibitem{luecking1981inequalities}
D.~H. Luecking.
\newblock Inequalities on {B}ergman spaces.
\newblock {\em Illinois J. Math.}, 25(1):1--11, 1981.

\bibitem{luecking1984closed}
D.~H. Luecking.
\newblock Closed ranged restriction operators on weighted {B}ergman spaces.
\newblock {\em Pacific J. Math.}, 110(1):145--160, 1984.

\bibitem{luecking1985forward}
D.~H. Luecking.
\newblock Forward and reverse {C}arleson inequalities for functions in
  {B}ergman spaces and their derivatives.
\newblock {\em Amer. J. Math.}, 107(1):85--111, 1985.

\bibitem{nazarov2004lower}
F.~Nazarov, M.~Sodin, and A.~Volberg.
\newblock Lower bounds for quasianalytic functions, {I}. {H}ow to control
  smooth functions.
\newblock {\em Math. Scand.}, 95(1):59--79, 2004.

\bibitem{peloso1994}
M.~M. Peloso.
\newblock Hankel operators on weighted {B}ergman spaces on strongly
  pseudoconvex domains.
\newblock {\em Illinois J. Math.}, 38(2):223--249, 1994.

\bibitem{ronkin-book}
L.~I. Ronkin.
\newblock {\em Introduction to the theory of entire functions of several
  variables}, volume~44 of {\em Translations of Mathematical Monographs}.
\newblock Amer. Math. Soc., Providence, RI, 1974.

\bibitem{xiawang2021}
Y.~Wang and J.~Xia.
\newblock Essential commutants on strongly pseudo-convex domains.
\newblock {\em J. Funct. Anal.}, 280(1):108775, 56, 2021.

\end{thebibliography}

\affil{Department of Mathematics, Washington University in Saint Louis\\ \indent 1 Brookings Drive, Saint Louis, MO 63130, USA}
\email{awgreen@wustl.edu, nathanawagner@wustl.edu}
 
\end{document}